\documentclass[12pt]{amsart}
\usepackage{amsmath, amssymb, amsthm, mathrsfs, mathtools, amscd}
\usepackage[latin1]{inputenc}
\usepackage{hyperref}
\usepackage[all]{xy}

\newtheorem{theorem}{Theorem}[section]
\newtheorem{lemma}[theorem]{Lemma}
\newtheorem{corollary}[theorem]{Corollary}
\newtheorem{proposition}[theorem]{Proposition}
\newtheorem{definition}[theorem]{Definition}
\newtheorem{remark}[theorem]{Remark}

\newcommand\on[1]{\operatorname{#1}}
\newcommand\mc[1]{\mathcal{#1}}
\newcommand\ps[1]{\underline{#1}}

\newcommand\ld{\lambda}

\newcommand{\id}{\on{id}}
\newcommand{\sa}{\on{sa}}
\newcommand{\op}{\on{op}}              

\newcommand{\hA}{\hat{A}}
\newcommand\hB{\hat{B}}

\newcommand{\hP}{\hat{P}}
\newcommand\hQ{\hat{Q}}
\newcommand{\cH}{\mathcal{H}}
\newcommand\BH{\mc{B}(\cH)}

\newcommand\eq[1]{(\ref{#1})}

\newcommand{\Sig}{\ps{\Sigma}}            

\newcommand\Set{\mathbf{Set}}                    
\newcommand\cN{\mc{N}}
\newcommand\cM{\mc{M}}
\newcommand\cV{\mc{V}}

\newcommand\VN{\mc{V}(\cN)}
\newcommand\VM{\mc{V}(\cM)}
\newcommand\SetC[1]{\Set^{#1^{\op}}}
\newcommand\SetVNop{\SetC{\VN}}
\newcommand\SetVMop{\SetC{\VM}}
\newcommand\SetVN{\Set^{\VN}}
\newcommand\SetVM{\Set^{\VM}}
\newcommand\SetCAop{\SetC{\CA}}
\newcommand\SetCA{\Set^{\CA}}
\newcommand\SetCB{\Set^{\CB}}
\newcommand\CA{\mc{C}(\cA)}

\newcommand\bbC{\mathbb{C}}
\newcommand\bbR{\mathbb{R}}
\newcommand\PN{\mc{P}(\cN)}
\newcommand\PM{\mc{P}(\cM)}
\newcommand\vNa{\mathbf{vNa}}

\newcommand\cA{\mc{A}}
\newcommand\cB{\mc{B}}
\newcommand\Ob[1]{\on{Ob}(#1)}

\newcommand\ra{\rightarrow}
\newcommand\mt{\mapsto}
\newcommand\lra{\longrightarrow}
\newcommand\lmt{\longmapsto}
\newcommand\hra{\hookrightarrow}

\newcommand\bjoin{\bigvee}
\newcommand\bmeet{\bigwedge}
\newcommand\tphi{\tilde\phi}
\newcommand\CB{\mc{C(B)}}
\newcommand\SetCBop{\SetC{\CB}}
\newcommand\SigA{\Sig^\cA}
\newcommand\SigB{\Sig^\cB}
\newcommand\SigN{\Sig^{\cN}}
\newcommand\SigM{\Sig^{\cM}}
\newcommand\cG{\mc G}

\newcommand\Id{\on{Id}}
\newcommand\Aut{\on{Aut}}
\newcommand\ol[1]{\overline{#1}}

\newcommand\io{\iota}
\newcommand\cJ{\mc J}

\newcommand\cC{\mc{C}}
\newcommand\ga{\gamma}
\newcommand\Ga{\Gamma}
\newcommand\pair[2]{\langle #1,#2\rangle}
\newcommand\fA{\ol{\cA}}
\newcommand\fB{\ol\cB}
\newcommand\fM{\ol\cM}
\newcommand\fN{\ol\cN}
\newcommand\cD{\mc D}
\newcommand\Pre[1]{\mathbf{Presh}(#1)}
\newcommand\Cpr[1]{\mathbf{Copresh}(#1)}
\newcommand\ucC{\mathbf{ucC^*}}
\newcommand\KHaus{\mathbf{KHaus}}
\newcommand\uC{\mathbf{uC^*}}
\newcommand\Dcpo{\mathbf{Dcpo}}
\newcommand\Arr[1]{\on{Arr}(#1)}
\newcommand\tT{\tilde T}
\newcommand\bbN{\mathbb{N}}

\newcommand\uCpart{\mathbf{uC^*part}}
\newcommand\VNpart{\mathbf{VNpart}}
\newcommand\Pos{\mathbf{Pos}}
\newcommand\cOML{\mathbf{cOML}}
\newcommand\JBW{\mathbf{JBW}}
\newcommand\JB{\mathbf{JB}}

\begin{document}

\title[Generalised Gelfand Spectra]{Generalised Gelfand Spectra of\\Nonabelian Unital $C^*$-Algebras}

\author{Andreas D\"oring}

\address{Andreas D\"oring, Clarendon Laboratory, Department of Physics, University of Oxford, Parks Road, OX1 3PU, Oxford, UK}
\email{doering@atm.ox.ac.uk}

\begin{abstract}
To each unital $C^*$-algebra $\cA$ we associate a presheaf $\SigA$, called the \emph{spectral presheaf of $\cA$}, which can be regarded as a generalised Gelfand spectrum. We develop a categorical notion of local duality and show that there is a contravariant functor from the category of unital $C^*$-algebras to a suitable category of presheaves containing the spectral presheaves. We clarify how much algebraic information about a $C^*$-algebra is contained in its spectral presheaf. A nonabelian unital $C^*$-algebra $\cA$ that is neither isomorphic to $\bbC^2$ nor to $\mc B(\bbC^2)$ is determined by its spectral presheaf up to quasi-Jordan isomorphisms. For a particular class of unital $C^*$-algebras, including all von Neumann algebras with no type $I_2$ summand, the spectral presheaf determines the Jordan structure up to isomorphisms.
\end{abstract}

\maketitle

\section{Introduction}
The famous Gelfand-Naimark theorem  \cite{GelNai43} states -- in modern language -- that there is a duality between the category of commutative $C^*$-algebras and $*$-homomorphisms and the category of locally compact Hausdorff spaces and proper continuous maps. The subcategory $\ucC$ of unital commutative $C^*$-algebras and unital $*$-homomorphisms is dual to the category $\KHaus$ of compact Hausdorff spaces and continuous maps. Given a unital commutative $C^*$-algebra $\cA$, the Gelfand spectrum $\Sigma(\cA)$ is the set of characters of $\cA$ (i.e., algebra homomorphisms $\ld:\cA\ra\bbC$), equipped with the topology of pointwise convergence. If $\phi:\cA\ra\cB$ is a unital $*$-homomorphism, then $\Sigma(\phi):\Sigma(\cB)\ra\Sigma(\cA)$ acts by precomposition (or `pullback'): each $\ld\in\Sigma(\cB)$ is mapped to $\ld\circ\phi\in\Sigma(\cA)$. Conversely, if $X$ is a compact Hausdorff space, then $C(X)$, the set of continuous, complex-valued functions on $X$, equipped with the supremum norm, is a commutative $C^*$-algebra under the pointwise algebraic operations. If $f:X\ra Y$ is a continuous function between compact Hausdorff spaces, then $C(f):C(Y)\ra C(X)$ also acts by precomposition (pullback): each $g\in C(Y)$ is mapped to $g\circ f\in C(X)$. The duality can be displayed as
\begin{equation}
			\xymatrix{\ucC \ar@<1ex>^-{\Sigma}[rr] \ar@{}|-{\bot}[rr] && \KHaus^{\op}. \ar@<1ex>^-{C(-)}[ll]}
\end{equation}
Gelfand-Naimark duality provides an enormously useful bridge between algebra on the one side and topology and geometry on the other. Algebraic notions have topological counterparts and vice versa, the simplest example being the correspondence between maximal ideals in a commutative $C^*$-algebra $\cA$ and points of its Gelfand spectrum $\Sigma(\cA)$. 

Of course, historically the question quickly arose if and how Gelfand-Naimark duality can be generalised to nonabelian $C^*$-algebras. There are a number of different approaches: instead of the space of characters as in the commutative case, one can consider the space of pure states of a $C^*$-algebra and regard it as a (generalised) spectrum on which one can try to define a functional representation of the algebra. This approach was pioneered by Kadison \cite{Kad51b,Kad52} and further developed by Akemann \cite{Ake71}, Akemann and Shultz \cite{AkeShu85}, Fujimoto \cite{Fuj98} and others. Another approach is based on sectional representations in $C^*$-bundles over the primitive ideal space, started by Fell \cite{Fel61} and leading to the Dauns-Hofmann theorem \cite{DauHof68}. Takesaki developed an approach to noncommutative Gelfand duality based on representation spaces and operator fields \cite{Tak67}, see also the related work by Kruszy\'nski and Woronowicz \cite{KruWor82}.

Noncommutative geometry is partly inspired by Gelfand-Naimark duality: the fact that topological notions have algebraic counterparts leads naturally to the idea of `translating' topology into commutative algebra first, and then to generalise to noncommutative algebras, which at least in spirit correspond to (algebras of functions on) noncommutative spaces. One can then consider further geometric structures, e.g. differentiable structure, suitably phrased in terms of noncommutative algebra. This idea has led to a rich body of deep and beautiful work, see e.g. Connes, Marcolli et al. from a differential geometric perspective \cite{Con94,ConMar08}; Manin, Majid et al. from deformation and quantum groups \cite{Man88,Maj95}; Rosenberg, Kontsevich et al. from algebraic geometry \cite{KonRos00}; also Hrushovski, Zilber et al. from geometric model theory \cite{HruZil96}.

In this article, we will start building towards a duality result for nonabelian $C^*$-algebras, but here we will focus on a somewhat more modest task: we will show how to define a new kind of spectrum of nonabelian unital $C^*$-algebras and von Neumann algebras that arises as a straightforward generalisation of the Gelfand spectrum of an abelian algebra. This generalised Gelfand spectrum will be a presheaf, called the \emph{spectral presheaf}. Instead of a single compact Hausdorff space as in classical Gelfand-Naimark duality, we have a presheaf of compact Hausdorff spaces. The spectral presheaf of a unital $C^*$-algebra $\cA$ is denoted $\SigA$. This object was first defined in the \emph{topos approach to quantum theory}, which was initiated by Isham and Butterfield \cite{IshBut98,IHB00} and substantially developed mainly by Isham and the author \cite{DoeIsh08a,DoeIsh08b,DoeIsh08c,DoeIsh08d,Doe09a,Doe09b,DoeIsh11,Doe11a,Doe11b,DoeIsh12,DoeBar12,Doe12} and by Heunen, Landsman, Spitters and Wolters \cite{HLS09a,HLS09b,HLS11,Wol10}.

The spectral presheaf is interpreted physically as a generalised state space of a quantum system. It was shown in \cite{IshBut98,IHB00} that the spectral presheaf of the algebra $\BH$ has no global sections if $\dim H\geq 3$, which is equivalent to the Kochen-Specker theorem \cite{KocSpe67}, an important theorem in the foundations of quantum theory. Later, this was generalised to arbitrary von Neumann algebras with no type $I_2$ summand \cite{Doe05}. In the present article, physical considerations will play no role. Mathematically, global sections of a presheaf are the analogues of points, so the spectral presheaf of a von Neumann algebra has no points, which can be seen as expressing its noncommutative character.

We will define the spectral presheaf of a unital $C^*$-algebra in section \ref{Sec_SpecPresheaves} and will show in section \ref{Sec_MorsAndMaps} that every unital $*$-homomorphism between unital $C^*$-algebras gives rise to a morphism between their presheaves in the opposite direction (Prop. \ref{Prop_AlgMorGivesSpecPreshMor}). The restriction to unital algebras is technically convenient, but does not seem essential. In future work, we will treat the non-unital case also.

We give a clear categorical underpinning to this construction in section \ref{Sec_PreshCopreshLocDuality} and introduce a notion of `local duality', which also connects the present article to the work by Heunen, Landsman, Spitters and Wolters. In particular, we show that there is a contravariant functor from the category of unital $C^*$-algebras and unital $*$-homomorphisms to a suitable category of presheaves in which the spectral presheaves of the algebras lie (Thm. \ref{Thm_TwoFunctors}). This presheaf category extends to a category of topoi. The construction presented here generalises the contravariant Gelfand-Naimark correspondence between unital abelian $C^*$-algebras and compact Hausdorff spaces to nonabelian $C^*$-algebras and suitable generalised spaces (although we will defer the discussion of topologies on the spectral presheaf and of continuity to future work). Analogous results hold for von Neumann algebras. 

In section \ref{Sec_AutomsOfSpecPresh}, we focus on isomorphisms between spectral presheaves and determine how much algebraic information about a nonabelian $C^*$-algebra is contained in its spectral presheaf. For a von Neumann algebra $\cN$ without type $I_2$ summand, we show that the spectral presheaf $\SigN$ determines exactly the Jordan $*$-structure of $\cN$. This is one of the main results (Thm. \ref{Thm_VNAs}). Using recent work by Hamhalter, we show that for a unital $C^*$-algebra $\cA$, the spectral presheaf determines $\cA$ up to quasi-Jordan $*$-isomorphism (Thm. \ref{Thm_SpecPresheafDeterminesPartialUnitalCStarAlg}) and, for a large class of $C^*$-algebras, also up to Jordan $*$-isomorphim (Thm. \ref{Thm_SpecPresheafOftenDeterminesUnitalCStarAlg}). 

In the article \cite{Doe12c}, we consider an application of the results in the present article in physics and develop further mathematical aspects. The spectral presheaf is considered as a state space for a quantum system, and flows on the spectral presheaf and on associated structures are defined. These flows allow describing the Schr\"odinger picture and the Heisenberg picture of time evolution of the quantum system in a new, more geometric manner than in standard quantum theory.

This work connects aspects of the theory of $C^*$-algebras and von Neumann algebras with aspects of category and topos theory. Standard references on operator algebras are e.g. \cite{KadRin83/86,Bla06}, and \cite{McLMoe92,Joh02/03} on topos theory.

\section{Unital $C^*$-algebras and their spectral presheaves}			\label{Sec_SpecPresheaves}
Let $\uC$ be the category of unital $C^*$-algebras with unital $*$-homo- morphisms as arrows. The category $\ucC$ of unital abelian $C^*$-subalg- ebras with unital $*$-homomorphisms as arrows is a full and faithful subcategory of $\uC$.

\begin{definition}			\label{Def_ContextCat}
Let $\cA\in\Ob\uC$ be a unital $C^*$-algebra, and let $\CA$ be the set of unital abelian $C^*$-subalgebras of $\cA$ that share the unit element $1$ with $\cA$. Equipped with the partial order given by inclusion, $\CA$ is called the \emph{context category of $\cA$}.
\end{definition}

By convention, we include the trivial subalgebra $C_0:=\bbC 1$ in $\CA$.\footnote{In some previous articles, e.g. in \cite{Doe12}, $C_0$ was excluded from $\CA$. For our purposes here, it makes sense to include it.}

\begin{remark}
Physically, the contexts, i.e., the elements of $C\in\CA$, are interpreted as `classical perspectives' on the quantum system. Each context $C$ determines and is determined by a set of commuting self-adjoint operators, which physically correspond to co-measurable physical quantities. The trivial abelian $C^*$-algebra $\bbC\hat1$ represents the trivial classical perspective. We will not be concerned with physical interpretation in the following. In the article \cite{Doe12c}, we will treat time evolution of quantum systems.
\end{remark}

The poset $\CA$ has all non-empty meets (greatest lower bounds). For any non-empty family $(C_i)_{i\in I}\subseteq\CA$,
\begin{equation}
			\bmeet_{i\in I} C_i:=\bigcap_{i\in I}C_i.
\end{equation}
Moreover, $\CA$ has all directed joins, hence it is a directed complete partial order (dcpo). For any directed family $(C_i)_{i\in I}\subseteq\CA$, the directed join is given by the (abelian) $C^*$-algebra generated by the algebras $C_i$. It was shown in \cite{DoeBar12}, Prop. 5.25 that there is a functor
\begin{align}
			\cC:\uC &\lra \Dcpo\\			\nonumber
			\cA &\lmt \CA
\end{align}
from the category $\uC$ of unital $C^*$-algebras and unital $*$-homomorphisms to the category $\Dcpo$ of dcpos and Scott-continuous functions.

If $\cA$ is abelian, then $\cA$ is the top element of $\CA$. In this case, the empty meet exists in $\CA$ and is equal to $\cA$, and all joins exist in $\CA$, so $\CA$ is a complete lattice. If $\cA$ is nonabelian, then the maximal abelian $C^*$-subalgebras are the maximal elements in the poset $\CA$, but no top element exists, which is equivalent to saying that the empty meet does not exist in $\CA$.

We now introduce the spectral presheaf of a unital $C^*$-algebra, which is a generalisation of the Gelfand spectrum of a unital abelian $C^*$-algebra.

\begin{definition}			\label{Def_SpecPresh}
Let $\cA$ be a unital $C^*$-algebra $\cA$. The \emph{spectral presheaf $\Sig$ of $\cA$} is the presheaf over $\CA$ given
\begin{itemize}
	\item [(a)] on objects: for all $C\in\CA$, $\Sig_C:=\Sigma(C)$, the Gelfand spectrum of $C$, i.e., the set of multiplicative states (characters, algebra homomorphisms) $\ld:C\ra\bbC$, equipped with the Gelfand topology (that is, the relative weak$^*$-topology when $\Sigma(C)$ is seen as a subset of the dual $C^*$ of $C$),
	\item [(b)] on arrows: for all inclusions $i_{C'C}:C'\hookrightarrow C$,
	\begin{align}
				\Sig(i_{C'C}):\Sig_C &\lra \Sig_{C'}\\			\nonumber
				\ld &\lmt \ld|_{C'}
	\end{align}
	the canonical restriction map. This map is surjective and continuous with respect to the Gelfand topologies.
\end{itemize}
\end{definition}

As is well-known, the Gelfand spectrum $\Sigma(C)$ of a unital $C^*$-algebra $\cA$ is a compact Hausdorff space. The spectral presheaf $\Sig$ of a---generally nonabelian---$C^*$-algebra $\cA$ consists of the Gelfand spectra of the abelian $C^*$-subalgebras of $\cA$, `glued together' in the canonical manner. If $\cA$ is abelian, then the component $\Sig_{\cA}$ at the top element of $\CA$ is the Gelfand spectrum of $\cA$.

The spectral presheaf was first defined by Isham and Butterfield (\cite{IshBut98}, and in \cite{IHB00} for von Neumann algebras) and used extensively in the so-called \emph{topos approach} to quantum theory (see \cite{DoeIsh11}). Being a presheaf, $\Sig$ is an object in the topos $\SetCAop$ of presheaves over the context category $\CA$. 

\section{Algebra morphisms, geometric morphisms and maps between spectral presheaves}			\label{Sec_MorsAndMaps}
We present the basic construction showing that every unital $*$-homo- morphism $\phi:\cA\ra\cB$ between unital $C^*$-algebras induces a map $\pair{\Phi}{\cG_\phi}:\SigB\ra\SigA$ in the opposite direction between the spectral presheaves of the algebras. 

Let $\cA,\,\cB\in\Ob\uC$ be unital $C^*$-algebras, and let $\phi\in\uC(\cA,\cB)$, that is, a unital $*$-homomorphism from $\cA$ to $\cB$. The algebra morphism $\phi$ induces a map
\begin{align}
			\tphi:\CA &\lra \CB\\			\nonumber
			C &\lmt \phi|_C(C)	
\end{align}
between the posets of unital abelian $C^*$-subalgebras of $\cA$ and $\cB$, respectively. For $C\in\CA$, the image $\phi|_C(C)$ is norm-closed (and hence a $C^*$-algebra) since $\phi$ is a $*$-homomorphism. The map $\tphi$ preserves the partial order and hence is well-defined.

\begin{remark}
The poset $\CA$ can be equipped with different topologies. One of them is the lower Alexandroff topology for which all lower sets in $\CA$ are open sets. Let $\phi:\cA\ra\cB$ be a unital $*$-homomorphism between unital $C^*$-algebras, and let $\tphi:\CA\ra\CB$ be the induced map between the context categories. The fact that $\tphi$ is order-preserving (monotone) directly implies that $\tphi$ is continuous with respect to the Alexandroff topologies on $\CA$ and $\CB$. Moreover, it is well-known that the topos of presheaves over $\CA$ is isomorphic to the topos of sheaves over $\CA_{Alex}$, the poset $\CA$ equipped with the Alexandroff topology,
\begin{equation}
			\SetCAop\simeq\on{Sh}(\CA_{Alex}).
\end{equation}
Another natural topology to consider is the Scott topology. As mentioned above, the morphism $\tphi:\CA\ra\CB$ is continuous with respect to the Scott topologies on $\CA$ and $\CB$. In the following, we will not be concerned with topologies on $\CA$, hence we will discuss presheaves and not sheaves.
\end{remark}

$\CA$ is the base category of the presheaf topos $\SetCAop$, $\CB$ is the base category of $\SetCBop$, and $\tphi:\CA\ra\CB$ is a morphism (i.e., functor) between these base categories. As is well-known, such a morphism between the base categories induces an essential geometric morphism
\begin{equation}
			\Phi:\SetCAop \lra \SetCBop
\end{equation}
between the topoi. (For some background on geometric morphisms, see e.g. \cite{McLMoe92,Joh02/03}; for essential geometric morphisms, see in particular section A4.1 in \cite{Joh02/03}.) 

\begin{remark}			\label{Rem_EssGeomMorFromFunctors}
$\SetCAop$ and $\SetCBop$ are presheaf categories over small base categories, and $\CB$ is Cauchy-complete, so every essential geometric morphism
\begin{align}
	f:\SetCAop\ra \SetCBop
\end{align}
arises from a functor $\tilde f:\CA\ra\CB$ by \cite{Joh02/03}, Lemma 4.1.5.\footnote{The lemma is formulated for functor categories, but the switch to presheaf categories is trivial.} Since conversely, every functor $\tilde f:\CA\ra\CB$ induces an essential geometric morphism $f:\SetCAop\ra\SetCBop$, there is a bijective correspondence between functors between the base categories and essential geometric morphisms between the presheaf topoi. This fact will be used throughout.
\end{remark}

The geometric morphism $\Phi$ has a direct image part
\begin{equation}
			\Phi_*:\SetCAop \lra \SetCBop
\end{equation}
in covariant direction, and an inverse image part
\begin{align}			\label{Eq_ActionOfInvImageFunctor}
			\Phi^*:\SetCBop &\lra \SetCAop\\			\nonumber
			\ps P &\lmt \ps P\circ\tphi
\end{align}
in contravariant direction. We display how the inverse image functor $\Phi^*$ acts on a presheaf here, since we will need this in the following (while the action of the direct image functor $\Phi_*$ will not play a role). $\Phi^*$ is left adjoint to $\Phi_*$, and $\Phi^*$ preserves finite limits. Since $\Phi$ is an essential geometric morphism, $\Phi^*$ also has a left adjoint $\Phi_!$ in covariant direction, that is, $\Phi_!:\SetCAop\ra\SetCBop$. This implies that the inverse image functor $\Phi^*$ preserves all limits as well as colimits.

Let $\SigB$ be the spectral presheaf of $\cB$. This is an object in $\SetCBop$. We use the inverse image functor $\Phi^*$ to map $\SigB$ to an object $\Phi^*(\SigB)$ in $\SetCAop$: by \eq{Eq_ActionOfInvImageFunctor}, we have
\begin{equation}
			\forall C\in\CA: \Phi^*(\SigB)_C = \SigB_{\tphi(C)}.
\end{equation}
The restriction maps of the presheaf $\Phi^*(\SigB)$ are given as follows: if $C',C\in\CA$ such that there is an inclusion $i_{C'C}:C'\hookrightarrow C$, then
\begin{align}
			\Phi^*(\SigB)(i_{C'C}):\Phi^*(\SigB)_C &\lra \Phi^*(\SigB)_{C'}\\			\nonumber
			\ld &\lmt \ld|_{\phi(C')}.
\end{align}

For each $C\in\CA$, we have a morphism
\begin{equation}
			\phi|_C:C \lra \phi(C)
\end{equation}
of unital abelian $C^*$-algebras. By Gelfand duality, this induces a continuous map
\begin{align}
			\cG_{\phi;C}:\Sigma(\phi(C)) &\lra \Sigma(C)\\			\nonumber
			\ld &\lmt \ld\circ\phi|_C
\end{align}
in the opposite direction between the Gelfand spectra. Noting that $\Phi^*(\SigB)_C=\SigB_{\tphi(C)}=\Sigma(\phi(C))$ and $\SigA_C=\Sigma(C)$, we have a map
\begin{equation}
			\cG_{\phi;C}:\Phi^*(\SigB)_C \lra \SigA_C,
\end{equation}
for each $C\in\CA$. 

Let $C,C'\in\CA$ such that $C'\subset C$, and let $\ld\in\Phi^*(\SigB)_C$. Then
\begin{align}
			\cG_{\phi;C'}(\Phi^*(\SigB)(i_{C'C})(\ld)) &= \cG_{\phi;C'}(\ld|_{\phi(C')})\\
			&= \ld|_{\phi(C')}\circ\phi|_{C'}\\
			&= (\ld\circ\phi|_C)|_{C'}\\
			&= \SigA(i_{C'C})(\ld\circ\phi|_C)\\
			&= \SigA(i_{C'C})(\cG_{\phi;C}(\ld)),
\end{align}
so
\begin{equation}
			\cG_{\phi;C'}\circ\Phi^*(\SigB)(i_{C'C}) = \SigA(i_{C'C})\circ\cG_{\phi;C}
\end{equation}
and the following diagram commutes for all $C',C\in\CA$ such that $C'\subset C$:
\[
			\xymatrix{
			\Phi^*(\SigB)_C  \ar[rr]^{\cG_{\phi;C}} \ar[dd]_{\Phi^*(\SigB)(i_{C'C})} & & \SigA_C \ar[dd]^{\SigA(i_{C'C})}
			\\ & & \\ 
			\Phi^*(\SigB)_{C'} \ar[rr]^{\cG_{\phi;C'}} & & \SigA_{C'} 
			}
\]  
This means that the maps $\cG_{\phi;C}$, $C\in\CA$, are the components of a natural transformation
\begin{equation}
			\cG_{\phi}:\Phi^*(\SigB) \lra \SigA.
\end{equation}
Thus, $\cG_{\phi}$ is an arrow in the topos $\SetCAop$, mapping the inverse image $\Phi^*(\SigB)$ of $\SigB$, the spectral presheaf of $\cB$, into $\SigA$, the spectral presheaf of $\cA$.

In a two-step process, we have mapped the spectral presheaf $\SigB$ of $\cB$ into the spectral presheaf $\SigA$ of $\cA$,
\begin{equation}
			\SigB \overset{\Phi^*}{\lmt} \Phi^*(\SigB) \overset{\cG_{\phi}}{\lmt} \SigA.
\end{equation}
Note that the map $\cG_{\phi}\circ\Phi^*:\SigB\ra\SigA$ is in contravariant direction with respect to the algebra morphism $\phi:\cA\ra\cB$ that we started from. 

The `composite' $\cG_{\phi}\circ\Phi^*$ consists of the inverse image part of a geometric morphism between topoi, followed by an arrow in a topos. Hence, it is not a proper composite arrow in any category. But maps like $\cG_{\phi}\circ\Phi^*$---which, more conventionally, are also denoted $\pair{\Phi}{\cG_\phi}$---are well-known in the theory of ringed topoi, where an arrow $\pair{\Ga}{\io}$ is a geometric morphism $\Ga:\mc X\ra\mc Y$ between the topoi, together with a specified map (morphism of internal rings) $\io:\Ga^*R_{\mc Y}\ra R_{\mc X}$ from the inverse image of the ring object $R_{\mc Y}$ in the ringed topos $\mc Y$ to the ring object $R_{\mc X}$ in the ringed topos $\mc X$, see e.g. \cite{nLab}.

Summing up, we have shown:
\begin{proposition}			\label{Prop_AlgMorGivesSpecPreshMor}
Let $\cA,\,\cB$ be unital $C^*$-algebras, and let $\phi:\cA\ra\cB$ be a unital $*$-homomorphism. There is a canonical map
\begin{equation}
			\pair{\Phi}{\cG_\phi}=\cG_{\phi}\circ\Phi^*:\SigB \lra \SigA
\end{equation}
in the opposite direction between the associated spectral presheaves.
\end{proposition}

\section{Presheaves, copresheaves and local duality}			\label{Sec_PreshCopreshLocDuality}
In this section, we develop some categorical background to the construction presented in the previous section. In particular, we introduce categories of presheaves and copresheaves with values in a fixed category and define a notion of local duality.

This section also relates the present work to work by Heunen, Landsman, Spitters and Wolters \cite{HLS09a,HLS09b,HLS11,Wol10}. There is a more refined categorical description using the fact that copresheaves are fibered over presheaves \cite{Fun12}. This will be developed in future work with Jonathon Funk, Pedro Resende and Rui Soares Barbosa.

\subsection{Bohrification and partial $C^*$-algebras}
The following construction is due to Heunen, Landsman and Spitters \cite{HLS09a}:
\begin{definition}			\label{Def_Bohrification}
Let $\cA\in\Ob{\uC}$ be a unital $C^*$-algebra, and let $\CA$ be its context category. The \emph{Bohrification of $\cA$} is the tautological copresheaf $\fA$ over $\CA$ that is given
\begin{itemize}
	\item [(a)] on objects: $\forall C\in\CA:\fA_C:=C$,
	\item [(b)] on arrows: for all inclusions $i_{C'C}:C'\hookrightarrow C$,
	\begin{align}
				\fA(i_{C'C}):\fA_{C'}=C' &\lra \fA_C=C\\			\nonumber
				\hA &\lmt \hA,
	\end{align}
	that is, $\fA(i_{C'C})=i_{C'C}$.
\end{itemize}
\end{definition}
The copresheaf $\fA$ is an object in the topos $\SetCA$ of covariant functors from $\CA$ to $\Set$. Using results by Banaschewski and Mulvey \cite{BanMul97,BanMul00,BanMul00b,BanMul06}, one can show that $\fA$ is an \emph{abelian} $C^*$-algebra internally in the topos $\SetCA$; for details, see \cite{HLS09a,HLS09b,HLS11}. 

We briefly consider von Neumann algebras, since we will need the notions of context category, spectral presheaf, and Bohrification of a von Neumann algebra later on.

\begin{definition}			\label{Def_VNSigNfN}
Let $\cN$ be a von Neumann algebra, and let $\VN$ denote the set of abelian von Neumann subalgebras of $\cN$ that share the unit element with $\cN$. Equipped with inclusion as partial order, $\VN$ is called the \emph{context category of $\cN$.} The \emph{spectral presheaf $\SigN$ associated with $\cN$} is the presheaf over $\VN$ given
\begin{itemize}
	\item [(a)] on objects: for all $V\in\VN$, $\Sig_V:=\Sigma(V)$, the Gelfand spectrum of $V$,
	\item [(b)] on arrows: for all inclusions $i_{V'V}:V'\hookrightarrow V$,
	\begin{align}
				\Sig(i_{V'V}):\Sig_V &\lra \Sig_{V'}\\			\nonumber
				\ld &\lmt \ld|_{V'}
	\end{align}
	the canonical restriction map. This map is surjective, continuous, closed and open with respect to the Gelfand topologies. 
\end{itemize}
The spectral presheaf is an object in the topos $\SetVNop$ of presheaves over the context category $\VN$.

The \emph{Bohrification of $\cN$} is the tautological copresheaf $\fN$ over $\VN$ that is given
\begin{itemize}
	\item [(a)] on objects: $\forall V\in\VN:\fN_V:=V$,
	\item [(b)] on arrows: for all inclusions $i_{V'V}:V'\hookrightarrow V$,
	\begin{align}
				\fN(i_{V'V}):\fA_{V'}=V' &\lra \fN_V=V\\			\nonumber
				\hA &\lmt \hA,
	\end{align}
	that is, $\fN(i_{C'C})=i_{C'C}$.
\end{itemize}
The Bohrification $\fN$ is an object in the topos $\SetVN$ of copresheaves over $\VN$.
\end{definition}

We now return to unital $C^*$-algebras.

\begin{definition}
Let $\cA$ be a unital $C^*$-algebra. The normal elements in $\cA$, equipped with the involution inherited from $\cA$ and partial operations of addition and multiplication also inherited from $\cA$, but defined only for (arbitrary pairs of) commuting elements, form the \emph{partial unital $C^*$-algebra associated with $\cA$}, which we will denote by $\cA_{part}$.

Let $\cA,\,\cB$ be unital $C^*$-algebras, and let $\cA_{part},\,\cB_{part}$ be the associated unital partial $C^*$-algebras. A unital map
\begin{equation}
			T:\cA_{part}\ra \cB_{part}
\end{equation}
such that for all commuting elements $\hA,\,\hB$ of $\cA_{part}$ and all $a,b\in\bbC$, we have
\begin{itemize}
	\item [(1)]	$T(a\hA+b\hB) = aT(\hA)+bT(\hB)$,
	\item [(2)] $T(\hA\hB) = T(\hA)T(\hB)$,
	\item [(3)] $T(\hA^*) = T(\hA)^*$,
\end{itemize}
is called a \emph{morphism of unital partial $C^*$-algebras}. This defines the category $\uCpart$ of partial unital $C^*$-algebras. (We only admit objects of the form $\cA_{part}$ coming from a unital $C^*$-algebra $\cA$.) A bijective map $T:\cA_{part}\ra \cB_{part}$ such that both $T$ and $T^{-1}$ are morphisms of unital partial $C^*$-algebras is called an \emph{isomorphism}. An isomorphism $T:\cA_{part}\ra \cB_{part}$ is called an \emph{automorphism}. The automorphisms of $\cA_{part}$ form a group $\Aut_{part}(\cA_{part})$.
\end{definition}

Since every morphism $\phi:\cA\ra\cB$ in the category $\uC$ of unital $C^*$-algebras gives a morphism $\phi:\cA_{part}\ra\cB_{part}$ in $\uCpart$, there is a faithful inclusion $i:\uC\hra\uCpart$ which is the identity on objects (because we only admit objects in $\uCpart$ of the form $\cA_{part}$ for $\cA\in\Ob{\uC}$). The inclusion functor $i$ is not full in general.

Condition (2) in the definition above implies $T(\hA)T(\hB)=T(\hB)T(\hA)$, so $T$ preserves commutativity. The fact that $T$ preserves the involution (condition (3) above) implies that, for all unital abelian $C^*$-subalgebras $C\in\CA$, the image $T(C)$ is norm-closed and hence a unital abelian $C^*$-subalgebra of $\cA$.


\begin{remark}			\label{Rem_fAIsPartialAlg}
As mentioned above, the copresheaf $\fA$ can be shown to be an abelian $C^*$-algebra internally in the topos $\SetCA$ \cite{HLS09a,HLS09b,HLS11}. The fact that internally $\fA$ is abelian while the usual topos-external $C^*$-algebra $\cA$ may be non-abelian is due to the definition of the algebraic structure of $\fA$ internally in the topos $\SetCA$: basically speaking, one only considers algebraic operations within each abelian subalgebra $C\in\CA$, that is, only between commuting, normal operators. One simply ignores non-normal operators (which are not contained in any abelian $C^*$-subalgebra) and forgets addition and multiplication between non-commuting normal operators (which never lie in the same abelian subalgebra), and hence obtains an abelian algebra. Seen topos-externally, this algebra can be identified with the unital partial $C^*$-algebra $\cA_{part}$. Analogous remarks apply to the Bohrification $\fN$ of a von Neumann algebra and the partial von Neumann algebra $\cN_{part}$ (see Def. \ref{Def_PartialVNA} below). We will make use of this in section \ref{Sec_AutomsOfSpecPresh}.
\end{remark}

Van den Berg and Heunen discuss a number of aspects of partial $C^*$-algebras and their morphisms in \cite{vdBHeu10}. The corresponding notion for von Neumann algebras is:

\begin{definition}			\label{Def_PartialVNA}
Let $\cN$ be a von Neumann algebra. The \emph{partial von Neumann algebra $\cN_{part}$ associated with $\cN$} is the partial unital $C^*$-algebra associated with $\cN$. Let $\cM,\,\cN$ be von Neumann algebras, and let $\cM_{part},\,\cN_{part}$ be the associated partial von Neumann algebras. A unital map
\begin{equation}
			T:\cM_{part} \lra \cN_{part}
\end{equation}
that is normal, that is, ultraweakly continuous on (ultraweakly closed) commuting subsets of $\cM_{part}$, and such that for all commuting elements $\hA,\,\hB$ of $\cM_{part}$ and all $a,b\in\bbC$, we have
\begin{itemize}
	\item [(1)]	$T(a\hA+b\hB) = aT(\hA)+bT(\hB)$,
	\item [(2)] $T(\hA\hB) = T(\hA)T(\hB)$,
	\item [(3)] $T(\hA^*) = T(\hA)^*$,
\end{itemize}
is called a \emph{morphism of partial von Neumann algebras}. This defines the category $\VNpart$ of partial von Neumann algebras. (We only admit objects of the form $\cN_{part}$ coming from a von Neumann algebra $\cN$.) A bijective map $T:\cM_{part}\ra \cN_{part}$ such that both $T$ and $T^{-1}$ are morphisms of partial von Neumann algebras is called an \emph{isomorphism}. An isomorphism $T:\cM_{part}\ra \cN_{part}$ is called an \emph{automorphism}. The automorphisms of $\cN_{part}$ form a group $\Aut_{part}(\cN_{part})$.
\end{definition}

The fact that $T$ is normal on commuting subsets implies that, for all abelian von Neumann subalgebras $V\in\cV(\cM)$, the image $T(V)$ is an abelian von Neumann subalgebra of $\cN$. 

\subsection{Categories of presheaves and copresheaves and local duality}
\label{Subsec_CatsOfPreshsAndCopreshs}
As we already saw, in order to discuss algebra morphisms between different (nonabelian) algebras and the corresponding morphisms between their spectral presheaves in the opposite direction, we need to consider presheaves and copresheaves over different base categories, since different algebras $\cA,\,\cB$ have different context categories $\CA,\,\CB$ of abelian subalgebras. 

The following construction of a category of presheaves (respectively copresheaves) with varying base categories was suggested by Nadish de Silva \cite{deS12}:

\begin{definition}			\label{Def_CatOfPresheaves}
Let $\cD$ be a category. The category $\Pre\cD$ of $\cD$-valued presheaves has as its objects functors of the form $\ps P:\cJ\ra\cD^{\op}$, where $\cJ$ is a small category. Arrows are pairs
\begin{equation}
			\pair{H}{\io}: (\ps{\tilde P}:\tilde\cJ\ra\cD^{\op})\lra(\ps P:\cJ\ra\cD^{\op}),
\end{equation}
where $H:\cJ\ra\tilde\cJ$ is a functor and $\io:H^*\ps{\tilde P}\ra\ps P$ is a natural transformation in $(\cD^{\op})^{\cJ}$. Here, $H^*\ps{\tilde P}$ is the presheaf over $\cJ$ given by
\begin{equation}
			\forall J\in\cJ: H^*\ps{\tilde P}_J=\ps{\tilde P}_{H(J)}.
\end{equation}






Let $\ps P_i:\cJ_i\ra\cD^{\op}$, $i=1,2,3$, be three presheaves over different base categories. Given two composable arrows $\pair{H'}{\io'}:\ps P_3\ra\ps P_2$ and $\pair{H}{\io}:\ps P_2\ra\ps P_1$, the composite is $\pair{H'\circ H}{\io\circ\io'}:\ps P_3\ra\ps P_1$, where, for all $J\in\cJ_1$, the natural transformation $\io\circ\io'$ has components
\begin{equation}
			(\io\circ\io')_J=\io_J\circ\io'_{H(J)}:((H'\circ H)^*\ps P_3)_J=(\ps P_3)_{H'(H(J))} \ra (\ps P_2)_{H(J)} \ra (\ps P_1)_J.
\end{equation}
\end{definition}


Every object $\ps P:\cJ\ra\cD^{\op}$ in a presheaf category $\Pre{\cD}$ can be identified with a functor $\ps P:\cJ^{\op}\ra\cD$ and hence can also be regarded as an object in the category $\cD^{\cJ^{\op}}$. We will assume that the category $\cD$ embeds into $\Set$ (that is, there is a forgetful functor $\cD\ra\Set$), so we can think of $\ps P$ as an object in the topos $\Set^{\cJ^{\op}}$, too.

If we interpret a morphism $\pair{H}{\io}:\ps{\tilde P}\ra\ps P$ between presheaves $\ps{\tilde P}:\tilde\cJ^{op}\ra\cD$ and $\ps P:\cJ^{\op}\ra\cD$ as a morphism between the presheaf topoi $\Set^{\tilde\cJ^{\op}}$ and $\Set^{\cJ^{\op}}$, then it consists of the inverse image part $H^*$ of the essential geometric morphism $H:\Set^{\cJ^{\op}}\ra\Set^{\tilde\cJ^{\op}}$ induced by the functor $H:\cJ\ra\tilde\cJ$ between the base categories, and a natural transformation $\io:H^*\tilde{\ps P}\ra\ps P$. Note that $\io$ behaves contravariantly with respect to $H$.

In this way, $\Pre{\cD}$ extends to a category of presheaf topoi, with the particular kind of morphisms between the objects in the presheaf topoi described above.

Analogously, we define:
\begin{definition}			\label{Def_CatOfCopresheaves}
Let $\cC$ be a category. The category $\Cpr\cC$ of $\cC$-valued copresheaves has as its objects functors of the form $\ol Q:\cJ\ra\cC$, where $\cJ$ is a small category. Arrows are pairs
\begin{equation}
			\pair{I}{\theta}: (\ol Q:\cJ\ra\cC)\lra(\ol{\tilde Q}:\tilde\cJ\ra\cC),
\end{equation}
where $I:\cJ\ra\tilde\cJ$ is a functor and $\theta:\ol Q\ra I^*\ol{\tilde Q}$ is a natural transformation in $\cC^{\cJ}$. Here, $I^*\ol{\tilde Q}$ is the copresheaf over $\cJ$ given by
\begin{equation}
			\forall J\in\cJ: I^*\ol{\tilde Q}_J=\ol{\tilde Q}_{I(J)}.
\end{equation}
Let $\ol Q_i:\cJ_i\ra\cD^{\op}$, $i=1,2,3$, be three copresheaves over different base categories. Given two composable arrows $\pair{I}{\theta}:\ol Q_1\ra\ol Q_2$ and $\pair{I'}{\theta'}:\ol Q_2\ra\ol Q_3$, the composite is $\pair{I'\circ I}{\theta'\circ\theta}:\ol Q_1\ra\ol Q_3$, where, for all $J\in\cJ_1$, the natural transformation $\theta'\circ\theta$ has components
\begin{equation}
			(\theta'\circ\theta)_J=\theta'_{I(J)}\circ\theta_J:(\ol Q_1)_J \ra (\ol Q_2)_{I(J)} \ra (\ol Q_3)_{I'(I(J))}=((I'\circ I)^*(\ol Q_3))_J.
\end{equation}
\end{definition}

We assume that there is a forgetful functor from $\cC$ to $\Set$. Then the category $\Cpr{\cC}$ extends to a category of copresheaf topoi, with the morphisms $\pair{I}{\theta}:\ol Q\ra\ol{\tilde Q}$ between copresheaves $\ol Q:\cJ\ra\cC$ and $\ol{\tilde Q}:\tilde\cJ\ra\cC$ in different copresheaf topoi given by the inverse image part $I^*$ of the essential geometric morphism $I:\Set^{\cJ}\ra\Set^{\tilde\cJ}$ induced by the functor $I:\cJ\ra\tilde\cJ$, and a natural transformation $\theta:\ol Q\ra\ol{\tilde Q}$. For copresheaves, the natural transformation $\theta$ behaves covariantly with respect to $I$.

We note that the construction of the spectral presheaf $\SigA$ is based on `local' Gelfand duality: for each context $C\in\CA$, the component $\SigA_C$ is simply given by the Gelfand spectrum of $C$. Hence, we use standard Gelfand duality locally in each abelian part and glue the Gelfand spectra together into a presheaf in the canonical way. Hence, for each $C\in\CA$, the component $\SigA_C$ of the spectral presheaf is the spectrum of the component $\fA_C=C$ of the copresheaf $\fA$. 

We now want to formalise this situation of having a `local duality' between a copresheaf (e.g. of algebras) and a presheaf (e.g. of their spectra).

\begin{proposition}			\label{Prop_DualEquiv}
Let $\cC,\,\cD$ be two categories that are dually equivalent,
\begin{equation}
			\xymatrix{\cC \ar@<1ex>^-{f}[rr] \ar@{}|-{\bot}[rr] && \cD^{\op}. \ar@<1ex>^-{g}[ll]}
\end{equation}
Then there is a dual equivalence
\begin{equation}
			\xymatrix{\Cpr\cC \ar@<1ex>^-{F}[rr] \ar@{}|-{\bot}[rr] && \Pre{\cD}^{\op}, \ar@<1ex>^-{G}[ll]}
\end{equation}
which we call the \emph{local duality} based on the duality between $\cC$ and $\cD$.
\end{proposition}

\begin{proof}
We first show that $g:\cD^{\op}\ra\cC$ induces a functor $G:\Pre\cD^{\op}\ra\Cpr\cC$. On objects, $G$ is given as follows: let $\ps P:\cJ\ra\cD^{\op}$, then $G(\ps P):\cJ\ra\cC$ is defined by
\begin{itemize}
	\item [(a)] $\forall J\in\cJ: G(\ps P)_J:=g(\ps P_J)$,
	\item [(b)] $\forall a:J'\ra J: G(a):=g(\ps P(a)):G(\ps P)_{J'}\ra G(\ps P)_J$.
\end{itemize}
On arrows: let $\pair{H}{\io}:(\ps{\tilde P}:\tilde\cJ\ra\cD^{\op})\ra(\ps P:\cJ\ra\cD^{\op})$ be an arrow in $\Pre\cD$, then
\begin{equation}
			G(\pair{H}{\io}):=\pair{H}{\theta},
\end{equation}
where $\theta:G(\ps P)\ra H^*(G(\ps{\tilde P}))$ is given, for all $J\in\cJ$, by
\begin{equation}
			\theta_J:= g(\io_J): G(\ps P)_J \lra H^*(G(\ps{\tilde P}))_J.
\end{equation}
For this, note that $\io_J:(H^*\ps{\tilde P})_J=\ps{\tilde P}_{H(J)}\ra\ps P_J$, so
\begin{align}
	g(\io_J):g(\ps P_J)=G(\ps P)_J \ra g(\ps{\tilde P}_{H(J)})=G(\ps{\tilde P})_{H(J)}=H^*(G(\ps{\tilde P}))_J.
\end{align}
Analogously, $f:\cC\ra\cD^{\op}$ induces $F:\Cpr\cC\ra\Pre\cD$: let $\ol Q:\cJ\ra\cC$, then $F(\ol Q):\cJ\ra\cD^{\op}$ is defined
\begin{itemize}
	\item [(a)] $\forall J\in\cJ: F(\ol Q)_J:=f(\ol Q_J)$,
	\item [(b)] $\forall a:J'\ra J: F(a):=f(\ol Q(a)):F(\ol Q)_{J'}\ra F(\ol Q)_J$.
\end{itemize}
On arrows: let $\pair{I}{\theta}:(\ol Q:\cJ\ra\cA)\ra(\ol{\tilde Q}:\tilde J\ra\cC)$ be an arrow in $\Cpr\cC$, then
\begin{equation}
			F(\pair{I}{\theta}):=\pair{I}{\io},
\end{equation}
where $\io:I^*(F(\ol{\tilde Q}))\ra F(\ol Q)$ is given, for all $J\in\cJ$, by
\begin{equation}
			\io_J:= f(\theta_J): I^*(F(\ol{\tilde Q}))_J \lra F(\ol Q)_J.
\end{equation}
For this, note that $\theta_J:\ol Q_J \ra (I^*\ol{\tilde Q})_J=\ol{\tilde Q}_{I(J)}$, so
\begin{align}
	f(\theta_J):f(\ol{\tilde Q}_{I(J)})=F(\ol{\tilde Q})_{I(J)}=I^*(F(\ol{\tilde Q}))_J \ra f(\ol Q_J)=F(\ol Q)_J.
\end{align}
Since $f\dashv g$, we have natural transformations $\epsilon:f\circ g\ra \id_{\cD^{op}}$ (the counit of the adjunction) and $\io:\id_{\cC}\ra g\circ f$ (the unit of the adjunction).

Let $\ps P:\cJ\ra\cD^{\op}$ be a $\cD$-valued presheaf over $\cJ$. By definition, we have, for each $J\in\cJ$,
\begin{equation}
			(F(G(\ps P)))_J = f(G(\ps P)_J) = f(g(\ps P_J)),
\end{equation}
and naturality of $\epsilon:f\circ g\ra \id_{\cD^{\op}}$ guarantees that $F\circ G:\Pre\cD\ra\Pre\cD$ is a morphism, namely $F\circ G=\pair{\id_{\cD^{\op}}}{\epsilon}$. 

Analogously, $G\circ F=\pair{\id_{\cC}}{\io}:\Cpr\cC\ra\Cpr\cC$.
\end{proof}

In this kind of duality between $\cD$-valued presheaves and $\cC$-valued copresheaves using `local duality' between $\cD^{\op}$ and $\cC$, there are two reversals of direction: one in the variance of the functors (presheaves and copresheaves), the other in the direction of morphisms between them.

We remark that Prop. \ref{Prop_DualEquiv} holds for presheaves and copresheaves over arbitrary small base categories $\cJ$, not only for posets (which we will mostly consider in the rest of this paper). Moreover, Prop. \ref{Prop_DualEquiv} applies to arbitrary dually equivalent pairs of categories $\cC,\cD$, so one can consider any classical duality and extend it to a duality between $\cC$-valued copresheaves and $\cD$-valued presheaves, not only Gelfand duality (as we will do in the rest of this paper). E.g. in \cite{CanDoe13}, we will show how Stone duality for Boolean algebras can be extended to orthomodular lattices.

\subsection{Application to Gelfand duality}	
Gelfand duality is the dual equivalence
\begin{equation}			\label{Eq_GelfandDuality}
			\xymatrix{\ucC \ar@<1ex>^-{\Sigma}[rr] \ar@{}|-{\bot}[rr] && \KHaus^{\op}. \ar@<1ex>^-{C(-)}[ll]}
\end{equation}
between unital abelian $C^*$-algebras and compact Hausdorff spaces. This gives a dual equivalence
\begin{equation}			\label{Eq_GelfandDualityPreshCopresh}
			\xymatrix{\Cpr{\ucC} \ar@<1ex>^-{\Sigma}[rr] \ar@{}|-{\bot}[rr] && \Pre{\KHaus}^{\op}. \ar@<1ex>^-{C(-)}[ll]}
\end{equation}
	
Let $\cA$ be a (generally \emph{nonabelian}) unital $C^*$-algebra, and let $\CA$ be its context category (cf. Def. \ref{Def_ContextCat}). The Bohrification $\fA:\CA\ra\ucC$, given by $\fA_C=C$ for all $C\in\cA$, is an object in the category $\Cpr{\ucC}$. The spectral presheaf $\SigA$ of $\cA$, given by $\SigA_C=\Sigma(C)$, the Gelfand spectrum of $C$, is an object in $\Pre{\KHaus}$.
	
These two objects correspond to each other via \eq{Eq_GelfandDualityPreshCopresh},
\begin{equation}
			\Sigma(\fA) = \SigA :\CA\lra\KHaus^{\op}
\end{equation}
and
\begin{equation}
			C(\SigA)=\fA:\CA\lra\ucC.
\end{equation}
If $\cA,\,\cB$ are (nonabelian) unital $C^*$-algebras and $\phi:\cA\ra\cB$ is a unital $*$-homomorphism, then we obtain a monotone map
\begin{align}
			\tphi:\CA &\lra\CB\\			\nonumber
			C &\lmt \phi(C).
\end{align}
	
Define a natural transformation $\varphi:\fA\ra\tphi^*\fB$ by
\begin{equation}
			\forall C\in\CA: \varphi_C:=\phi|_C:\fA_C=C \lra (\tphi^*\fB)_C=\fB_{\phi(C)}=\phi(C).
\end{equation}
Then 
\begin{equation}
			\pair{\tphi}{\varphi}:\fA \lra \fB
\end{equation}
is a morphism in $\Cpr\ucC$. This is essentially Prop. 34 in \cite{vdBHeu10} (see version 3).\footnote{Van den Berg and Heunen make a different choice for the direction of their geometric morphism, which leads to an arrow between $\fA$ and $\fB$ in the opposite direction to the morphism $\phi:\cA\ra\cB$. Our choice, which gives an arrow in the same direction as the external morphism $\phi$, seems more natural.}
	
Dually, define a natural transformation $\cG_\phi:\tphi^*\SigB\ra\SigA$ by
\begin{equation}
			\forall C\in\CA: \cG_{\phi;C}:=\Sigma(\phi|_C):(\tphi^*\SigB)_C=\SigB_{\phi(C)}=\Sigma(\phi(C)) \lra \SigA_C=\Sigma(C).
\end{equation}
Then 
\begin{equation}			\label{Eq_MorSpecPreshs}
			\pair{\tphi}{\cG_\phi}:\SigB \lra \SigA
\end{equation}
is a morphism in $\Pre{\KHaus}$.

By construction, $\Sigma(\pair{\tphi}{\varphi})=\pair{\tphi}{\cG_\phi}$ and $C(\pair{\tphi}{\cG_\phi})=\pair{\tphi}{\varphi}$ (using the duality in \eq{Eq_GelfandDualityPreshCopresh}).

Summing up, we have shown:
\begin{theorem}			\label{Thm_TwoFunctors}
There is a contravariant functor
\begin{equation}
			\ps{\mc S}: \uC \lra \Pre{\KHaus}
\end{equation}
from the category of unital $C^*$-algebras to the category of compact Hausdorff space-valued presheaves, given
\begin{itemize}
	\item [(a)] on objects: $\forall\cA\in\Ob{\uC}:\ps{\mc S}(\cA):=\SigA$, the spectral presheaf of $\cA$,
	\item [(b)] on arrows: for all unital $*$-morphisms $(\phi:\cA\ra\cB)\in\Arr{\uC}$,
	\begin{align}
				\ps{\mc S}(\phi)=\pair{\tphi}{\cG_\phi}:\ps{\mc S}(\cB)\ra\ps{\mc S}(\cA).
	\end{align}
\end{itemize}

Moreover, there is a covariant functor
\begin{equation}
			\ol{\mc B}: \uC \lra \Cpr{\ucC}
\end{equation}
from the category of unital $C^*$-algebras to the category of unital abelian $C^*$-algebra-valued copresheaves, given
\begin{itemize}
	\item [(c)] on objects: $\forall\cA\in\Ob{\uC}:\ol{\mc B}(\cA):=\fA$, the Bohrification of $\cA$,
	\item [(d)] on arrows: for all unital $*$-morphisms $(\phi:\cA\ra\cB)\in\Arr{\uC}$,
	\begin{align}
				\ol{\mc B}(\phi)=\pair{\tphi}{\varphi}:\ol{\mc B}(\cA)\ra\ol{\mc B}(\cB).
	\end{align}
\end{itemize}

For each $\cA\in\Ob{\uC}$, the object $\ol{\mc B}(\cA)=\fA$ is an object in the category $\ucC^{\CA}$, which is a subcategory of the copresheaf topos $\SetCA$. By the `local duality' described in \eq{Eq_GelfandDualityPreshCopresh}, $\fA$ corresponds to $\SigA$, which is an object in the category $\KHaus^{\CA^{\op}}$ that is a subcategory of the presheaf topos $\SetCAop$.
\end{theorem}

As mentioned above, $\fA$ is an internal abelian $C^*$-algebra in $\SetCA$, as was shown in \cite{HLS09a}. Morphisms between presheaves (respectively copresheaves) with different base categories $\cJ,\,\tilde\cJ$, but the same codomain category $\cD$ (respectively $\cC$) are of the kind described in Def. \ref{Def_CatOfPresheaves} (respectively Def. \ref{Def_CatOfCopresheaves}).

Thm. \ref{Thm_TwoFunctors} should be compared with section \ref{Sec_MorsAndMaps}; in particular, the existence of the functor $\ps{\mc S}$ subsumes the result of Prop. \ref{Prop_AlgMorGivesSpecPreshMor}.


\section{Isomorphisms of spectral presheaves and Jordan structure}			\label{Sec_AutomsOfSpecPresh}
It is clear that local duality as in Prop. \ref{Prop_DualEquiv} can give information about the abelian parts of a nonabelian operator algebra, but in order to extract more information, we have to take a more global view. There are two obvious ways to extract more global information: by considering the base category $\CA$ of the spectral presheaf $\SigA$ of a $C^*$-algebra (or von Neumann algebra), and by considering isomorphisms of the spectral presheaves of two unital $C^*$-algebras, in order to determine how they relate to certain isomorphisms of the algebras.

We will first consider von Neumann algebras in subsection \ref{Subsec_vNas}, for which the following characterisation is possible: let $\cM,\,\cN$ be von Neumann algebras with no type $I_2$ summand. Then their spectral presheaves $\SigM,\,\SigN$ are isomorphic if and only if $\cM,\,\cN$ are Jordan $*$-isomorphic, if and only if their context categories $\VM,\,\VN$ are order-isomorphic. This is the content of Thm. \ref{Thm_VNAs}, which is one of the main results of this paper. As a corollary, we obtain that if a von Neumann algebra $\cN$ has no type $I_2$ summand, then the group of automorphisms of the spectral presheaf $\SigN$ of $\cN$ is contravariantly isomorphic to the group of Jordan $*$-automorphisms of $\cN$.

In subsection \ref{Subsec_CStarAlgs}, we characterise isomorphisms between the spectral presheaves of unital $C^*$-algebras $\cA,\,\cB$. Using a recent result by Hamhalter, we show that a $C^*$-algebra that is neither isomorphic to $\bbC^2$ nor to $\mc B(\bbC^2)$ is determined by its spectral presheaf $\SigA$ up to quasi-Jordan isomorphism (which is the same as up to isomorphism as a partial unital $C^*$-algebra), see Thm. \ref{Thm_SpecPresheafDeterminesPartialUnitalCStarAlg}. Furthermore, we will show that for a certain class of unital $C^*$-algebras, the spectral presheaf $\SigA$ determines an algebra $\cA$ up to Jordan $*$-isomorphisms, see Thm. \ref{Thm_SpecPresheafOftenDeterminesUnitalCStarAlg}.

Moreover, we show that there is an injective contravariant group homomorphism from the automorphism group of a unital $C^*$-algebra $\cA$ into the group $\Aut(\SigA)$ of automorphisms of its spectral presheaf, see Prop. \ref{Prop_RepOfAutcAInAutSigA}. This directly implies the analogous result for von Neumann algebras.

\subsection{Von Neumann algebras}			\label{Subsec_vNas}
\begin{definition}			\label{Def_IsomFromSigNToSigM}
Let $\cM,\,\cN$ be von Neumann algebras, and let $\SigM,\,\SigN$ be their spectral presheaves (see Def. \ref{Def_VNSigNfN}). An \emph{isomorphism from $\SigN$ to $\SigM$} is a pair $\pair{\Ga}{\io}$, where $\Ga:\SetVMop\ra\SetVNop$ is an essential geometric isomorphism, induced by an order-isomorphism $\ga:\VM\ra\VN$ (called the \emph{base map}). $\Ga^*:\SetVNop\ra\SetVMop$ is the inverse image functor of the geometric isomorphism $\Ga$, and $\io:\Ga^*(\SigN)\ra\SigM$ is a natural isomorphism for which each component $\io_W:(\Ga^*(\SigN))_W\ra\SigM_W$, where $W\in\VM$, is a homeomorphism. Hence, an isomorphism $\pair{\Ga}{\io}:\SigN\ra\SigM$ acts by
\begin{equation}
			\SigN \stackrel{\Ga^*}{\lra} \Ga^*(\SigN) \stackrel{\io}{\lra} \SigM.
\end{equation}
We will also use the notation $\io\circ\Ga^*$ for an isomorphism $\pair{\Ga}{\io}$ (although this `composition' is not a composition of morphisms in a single category, but the inverse image part of a geometric isomorphism, followed by a natural transformation). If $\cM=\cN$, an isomorphism $\pair{\Ga}{\io}:\SigN\ra\SigN$ is called an \emph{automorphism of $\SigN$}.
\end{definition}

\begin{lemma}			\label{Lem_SigNIsGroup}
The automorphisms of $\SigN$ form a group $\Aut(\SigN)$.
\end{lemma}

\begin{proof}
We define the group operation as
\begin{align}
			\Aut(\SigN)\times\Aut(\SigN) &\lra \Aut(\SigN)\\			\nonumber
			(\pair{\Ga_1}{\io_1},\pair{\Ga_2}{\io_2}) &\lmt \pair{\Ga_2\circ\Ga_1}{\io_1\circ\io_2},
\end{align}
where $\Ga_2\circ\Ga_1:\SetVNop\ra\SetVNop$ is the essential geometric automorphism induced by $\ga_2\circ\ga_1$, the composite of the base maps underlying $\Ga_2$ respectively $\Ga_1$, and $\io_1\circ\io_2$ is the natural isomorphism with components
\begin{equation}
			\forall V\in\VN: (\io_1\circ\io_2)_V=\io_{1;V}\circ\io_{2;\ga_1(V)}
\end{equation}
(cf. Def. \ref{Def_CatOfPresheaves}). The automorphism $\pair{\Id}{\id}$ acts as a neutral element. If $\pair{\Ga}{\io}$ is an automorphism of $\SigN$ with underlying base map $\ga:\VN\ra\VN$, then $\pair{\Ga^{-1}}{\io^{-1}}$ with underyling base map $\ga^{-1}$ is its inverse.
\end{proof}

An isomorphism in the sense of Def. \ref{Def_IsomFromSigNToSigM} from $\SigN$ as an object in the topos $\SetVNop$ to $\SigM$ as an object in the topos $\SetVMop$ corresponds to an isomorphism from $\SigN$ to $\SigM$ as objects in the category $\Pre{\KHaus}$, see Def. \ref{Def_CatOfPresheaves}. Conversely, each isomorphism from $\SigN$ to $\SigM$ in $\Pre{\KHaus}$ determines a unique isomorphism from $\SigN$ as an object of $\SetVNop$ to $\SigM$ as an object of $\SetVMop$.

Every poset $P$ can be seen as a category, with objects the elements $a,b,...\in P$ of the poset, and arrows expressing the order: there is an arrow $a\ra b$ if and only if $a\leq b$. An order-preserving map $P\ra Q$ is the same as a (covariant) functor from $P$ to $Q$. The requirement that the base map $\ga:\VM\ra\VN$ is an order-isomorphism is equivalent to saying that $\ga$ is an invertible covariant functor from $\VM$ to $\VN$, which is equivalent to $\ga$ being an isomorphism in the category $\Pos$ of posets. We could suppress any reference to the base map in Def. \ref{Def_IsomFromSigNToSigM} and just require that there is an essential geometric automorphism $\Ga:\SetVMop\ra\SetVNop$: each essential geometric morphism between presheaf topoi is induced by a (covariant) functor between the base categories of the topoi, that is, a functor $\ga:\VM\ra\VN$ in our case (cf. Rem. \ref{Rem_EssGeomMorFromFunctors}). Moreover, the fact that $\Ga$ is a geometric \emph{iso}morphism implies that the functor $\ga$ must be invertible.


The action of $\Ga^*$ on $\SigN$---and on any other presheaf in the topos---is analogous to the pullback of a bundle. Concretely, the component of $\Ga^*(\SigN)$ at $W\in\VM$ is given by $\Ga^*(\SigN)_W=\SigN_{\ga(W)}$, so we assign the Gelfand spectrum of $\ga(W)\in\VN$ to $W$. 

Note that if and only if $W$ and $\ga(W)$ are isomorphic as abelian $C^*$-algebras, there is a homeomorphism (i.e., an isomorphism in the category of topological spaces) $\io_W:\SigN_{\ga(W)}\ra\SigM_W$ between their spectra. The definition of an isomorphism from $\SigN$ to $\SigM$ requires that such an isomorphism $\io_W$ exists for every $W\in\VM$, and that the $\io_W$ are the components of a natural isomorphism $\io:\Ga^*(\SigN)\ra\SigN$. We will see in Thm. \ref{Thm_VNAs} that in fact, given $\ga$, the existence and uniqueness of $\io$ are guaranteed: every order-isomorphism $\ga:\VM\ra\VN$ induces a unique isomorphism $\pair{\Ga}{\io}:\SigN\ra\SigM$ of the spectral presheaves, and in particular, a unique natural isomorphism $\io:\Ga^*(\SigN)\ra\SigM$.

We present two results that actually will be proven in subsection \ref{Subsec_CStarAlgs} for the more general case of unital $C^*$-algebras, which is why we do not give the proofs here.

\begin{proposition}			\label{Prop_IsomsOfSpecPreshsAreJordanIsoms}
Let $\cM,\,\cN$ be von Neumann algebras. There is a bijective correspondence between isomorphisms $\pair{\Ga}{\io}:\SigN\ra\SigM$ of the spectral presheaves and isomorphisms $T:\cM_{part}\ra\cN_{part}$ of the associated partial von Neumann algebras. Hence, $\cM$ and $\cN$ are isomorphic as partial von Neumann algebras if and only if their spectral presheaves are isomorphic.
\end{proposition}

This result follows straightforwardly from the analogous result for unital $C^*$-algebras, see Prop. \ref{Prop_IsomsOfSpecPreshsGiveIsomsAsPartialAlgs} below (using Lemmas \ref{Lem_NatIsoGivesPartialAutom} and \ref{Lem_PartialMorGivesNatIso}).

\begin{corollary}			\label{Cor_ContravGroupIsoForVNAs}
Let $\cN$ be a von Neumann algebra. There is a contravariant group isomorphism
\begin{equation}
			A:\Aut(\SigN) \lra \Aut_{part}(\cN_{part})
\end{equation}
between $\Aut(\SigN)$, the group of automorphisms of the spectral presheaf of $\cN$, and $\Aut_{part}(\cN_{part})$, the group of automorphisms of the partial von Neumann algebra $\cN_{part}$.
\end{corollary}

\begin{proposition}			\label{Prop_RepOfAutcNInAutSigN}
Let $\cN$ be a von Neumann algebra, and let $\SigN$ be its spectral presheaf. Let $\Aut(\cN)$ be the automorphism group of $\cN$, and let $\Aut(\SigN)$ be the automorphism group of $\SigN$. There is an injective group homomorphism from $\Aut(\cN)$ to $\Aut(\SigN)^{op}$ (that is, there is an injective, contravariant group homomorphism from $\Aut(\cN)$ into $\Aut(\SigN)$), given by
\begin{align}
			\Aut(\cN) &\lra \Aut(\SigN)^{op}\\			\nonumber
			\phi &\lmt \pair{\Phi}{\cG_\phi}=\cG_{\phi}\circ\Phi^*,
\end{align}
where $\pair{\Phi}{\cG_\phi}$ is the automorphism of $\SigN$ induced by $\phi$ as in Section \ref{Sec_MorsAndMaps} (see also Thm. \ref{Thm_TwoFunctors}).
\end{proposition}
This follows easily from Prop. \ref{Prop_RepOfAutcAInAutSigA}, the analogous result for unital $C^*$-algebras.

\begin{definition}
Let $\cM,\,\cN$ be von Neumann algebras, and let $\fM,\,\fN$ be their Bohrifications (see Def. \ref{Def_VNSigNfN}). An \emph{isomorphism from $\fM$ to $\fN$} is a pair $\pair{\Ga}{\kappa}$, where $\Ga:\SetVM\ra\SetVN$ is an essential geometric isomorphism, induced by an order-isomorphism $\ga:\VM\ra\VN$ (the \emph{base map}). $\Ga^*:\SetVN\ra\SetVM$ is the inverse image functor of the geometric automorphism $\Ga$, and $\kappa:\fM\ra\Ga^*(\fM)$ is a natural isomorphism for which each component $\kappa_W:\fM_W\ra\Ga^*(\fN)_W$, $W\in\VM$, is a unital $*$-isomorphism (of abelian $C^*$-algebras). If $\cM=\cN$, an isomorphism $\pair{\Ga}{\kappa}:\fN\ra\fN$ is called an \emph{automorphism of $\fN$}.
\end{definition}

The automorphisms of $\fN$ form a group, which we denote as $\Aut(\fN)$.

An isomorphism from $\fM$ as an object of the topos $\SetVM$ to $\fN$ as an object of the topos $\SetVN$ in the sense defined above corresponds to an isomorphism from $\fM$ to $\fN$ in the category $\Cpr{\ucC}$ of copresheaves with values in unital abelian $C^*$-algebras (see Def. \ref{Def_CatOfCopresheaves}). Prop. \ref{Prop_DualEquiv} implies:

\begin{corollary}			\label{Cor_AutSigNContravIsomToAutofN}
Let $\cM,\,\cN$ be von Neumann algebras. Every isomorphism $\pair{\tilde\Ga}{\kappa}:\fM\ra\fN$ between their Bohrifications corresponds to an isomorphism $\pair{\Ga}{\io}:\SigN\ra\SigM$ in the opposite direction between their spectral presheaves. If $\cM=\cN$, there is a contravariant group isomorphism
\begin{equation}
			\Aut(\SigN) \lra \Aut(\fN)
\end{equation}
between the group of automorphisms of the spectral presheaf of $\cN$ and the group of automorphisms of the Bohrification of $\cN$.
\end{corollary}

Since the Bohrifications $\fM,\,\fN$ correspond to the (topos-external) partial von Neumann algebras $\cM_{part},\,\cN_{part}$ (cf. Rem. \ref{Rem_fAIsPartialAlg}), Prop. \ref{Prop_IsomsOfSpecPreshsAreJordanIsoms} is the `external version' of Cor. \ref{Cor_AutSigNContravIsomToAutofN}.

\begin{proposition}			\label{Prop_PartVNIsomsAreProjLattIsoms}
Let $\cM,\,\cN$ be von Neumann algebras, and let $\cM_{part}$, $\cN_{part}$ be the corresponding partial von Neumann algebras. There is a bijective correspondence between isomorphisms of complete orthomodular lattices
\begin{equation}
			\tT: \PM \lra \PN
\end{equation}
of the projection lattices of $\cM,\,\cN$ and isomorphisms
\begin{equation}
			T: \cM_{part} \lra \cN_{part}
\end{equation}
of partial von Neumann algebras. Hence, $\cM$ and $\cN$ are isomorphic as partial von Neumann algebras if and only if their projection lattices are isomorphic.
\end{proposition}

\begin{proof}
Let $T:\cM_{part}\ra \cN_{part}$ be an isomorphism of partial von Neumann algebras. Then, for all projections $\hP\in\PM\subset \cM_{part}$,
\begin{equation}
			T(\hP)=T(\hP^2)=T(\hP)^2,\qquad T(\hP)=T(\hP^*)=T(\hP)^*,
\end{equation}
so $T(\hP)$ is a projection, and $T:\PM\ra\PN$ is a bijection. If $\hP\leq\hQ$, then
\begin{equation}
			T(\hP)=T(\hP\hQ)=T(\hP)T(\hQ),
\end{equation}
so $T(\hP)\leq T(\hQ)$, that is, $T$ preserves the order. Let $T^{-1}:\cN_{part}\ra\cM_{part}$ be the inverse map of $T$ (which is an isomorphism of partial von Neumann algebras, too). Then, if $\hP,\,\hQ\in\PN$ such that $\hP\leq\hQ$,
\begin{equation}
			T^{-1}(\hP)=T^{-1}(\hP\hQ)=T^{-1}(\hP)T^{-1}(\hQ),
\end{equation}
so $T^{-1}(\hP)\leq T^{-1}(\hQ)$, and $T$ reflects the order. Moreover, for all $\hP\in\PM$,
\begin{equation}
			T(\hat 1-\hP)=T(\hat 1)-T(\hP)=\hat 1-T(\hP),
\end{equation}
so $T$ preserves complements. Let $(\hP_i)_{i\in I}\subseteq\PM$ be a family of projections, not necessarily commuting. Then
\begin{equation}
			\forall i\in I: T(\hP_i)\leq T(\bjoin_{i\in I}\hP_i),
\end{equation}
so
\begin{equation}			\label{Ineq_SimpleJoins}
			\bjoin_{i\in I}T(\hP_i)\leq T(\bjoin_{i\in I}\hP_i).
\end{equation}
Note that here, we use joins in $\PM$ respectively $\PN$ between not necessarily commuting projections, so we employ the lattice structure of $\PM$ respectively $\PN$ in our argument, and not just the partial algebra structures of $\cM_{part}$ respectively $\cN_{part}$. Since $T^{-1}$ also preserves and reflects the order, \eq{Ineq_SimpleJoins} is equivalent to
\begin{equation}
			T^{-1}(\bjoin_{i\in I}T(\hP_i)) \leq \bjoin_{i\in I}\hP_i.
\end{equation}
For the left hand side, we have
\begin{equation}
			T^{-1}(\bjoin_{i\in I}T(\hP_i)) \geq \bjoin_{i\in I}T^{-1}(T(\hP_i)) = \bjoin_{i\in I}\hP_i,
\end{equation}
because $T^{-1}$ preserves the order. Hence,
\begin{align}
			&T^{-1}(\bjoin_{i\in I}T(\hP_i)) = \bjoin_{i\in I}\hP_i\\
			\Longleftrightarrow\; &\bjoin_{i\in I}T(\hP_i) = T(\bjoin_{i\in I}\hP_i),
\end{align}
so $T$ preserves all joins. In a completely analogous fashion, one shows that
\begin{align}
			\bmeet_{i\in I}T(\hP_i)=T(\bmeet_{i\in I}\hP_i),
\end{align}
that is, $T$ also preserves all meets and hence 
\begin{equation}
			\tT:=T|_{\PM}
\end{equation}
is an isomorphism from the complete orthomodular lattice $\PM$ to the complete orthomodular lattice $\PN$. 

Conversely, given an isomorphism $\tT:\PM\ra\PN$ of the projection lattices, define a partial automorphism $T:\cM_{part}\ra\cN_{part}$ in the obvious way: for self-adjoint operators $\hA$ that are finite real-linear combinations of projections, we have
\begin{equation}
			T(\hA)=T(\sum_{i=1}^n a_i\hP_i):=\sum_{i=1}^n a_i T(\hP_i).
\end{equation}
An arbitrary self-adjoint $\hA$ can be approximated in norm by a family $(\hA_i)_{i\in\bbN}$ of self-adjoint operators $\hA_i$ that are finite real-linear combinations of projections, so
\begin{equation}
			T(\hA):=\lim_{i\ra\infty}T(\hA_i),
\end{equation}
where the limit is taken in the norm topology. For a non-self-adjoint normal operator $\hB\in \cN_{part}$, use the decomposition $\hB=\hA_1+i\hA_2$ into self-adjoint operators and let
\begin{equation}
			T(\hB):=T(\hA_1)+iT(\hA_2).
\end{equation}
It remains to show that this $T$ is an isomorphism from the partial von Neumann algebra $\cM_{part}$ to the partial von Neumann algebra $\cN_{part}$, that is, we need to show that $T$ preserves addition and multiplication of commuting normal operators. (Preservation of the unit element and preservation of the involution are obvious.)

Let $\hA,\,\hB\in\cM_{part}$ be commuting operators which both are finite real-linear combinations of projections, that is
\begin{equation}
			\hA=\sum_{i=1}^n a_i\hP_i,\qquad \hB=\sum_{i=1}^n b_i\hP_i.
\end{equation}
Note that the projections $(\hP_i)_{i=1,...,n}$ are pairwise orthogonal, and that $\hA$ and $\hB$ can be written as linear combinations of the \emph{same} projections. We have $\hA+\hB=\sum_{i=1}^n (a_i+b_i)\hP_i$, so
\begin{align}
			T(\hA+\hB) &= \sum_{i=1}^n (a_i+b_i)\tT(\hP_i)\\
			&= \sum_{i=1}^n a_i\tT(\hP_i) + \sum_{i=1}^n b_i\tT(\hP_i)\\
			&= T(\hA)+T(\hB).
\end{align}
Moreover, $\hA\hB=\sum_{i=1}^n a_i b_i \hP_i$, so
\begin{align}
			T(\hA\hB)=\sum_i a_i b_i \tT(\hP_i).
\end{align}
On the other hand,
\begin{align}
			T(\hA)T(\hB) &= \sum_i a_i\tT(\hP_i) \sum_j b_j\tT(\hP_j)\\
			&= \sum_i a_i (\tT(\hP_i)\sum_j b_j\tT(\hP_j))\\
			&= \sum_i a_i \delta_{ij} b_j \tT(\hP_j)\\
			&= \sum_i a_i b_i \tT(\hP_i),
\end{align}
so $T(\hA\hB)=T(\hA)T(\hB)$. By continuity in the norm-topology, $T$ can be extended to act as a partial automorphism on all self-adjoint operators in $\cM_{part}$, and by linearity to all normal operators. Hence, $T:\cM_{part}\ra\cN_{part}$ is an isomorphism of partial von Neumann algebras.

The two maps $T\mt\tT$ and $\tT\mt T$ are inverse to each other by construction, so there is a bijection between isomorphisms $\tT:\PM\ra\PN$ in $\cOML$, the category of complete orthomodular lattices, and isomorphisms $T:\cM_{part}\ra\cN_{part}$ in $\VNpart$, the category of partial von Neumann algebras.
\end{proof}

\begin{corollary}			\label{Cor_GroupIsoPartAutCOMLAut}
Let $\cN$ be a von Neumann algebra, and let $\cN_{part}$ be the corresponding partial von Neumann algebra. There is a group isomorphism
\begin{equation}
			B:\Aut_{cOML}(\PN) \lra \Aut_{part}(\cN_{part})
\end{equation}
between the group $\Aut_{cOML}(\PN)$ of automorphisms $\tT:\PN\ra\PN$ of the complete orthomodular lattice $\PN$ of projections in $\cN$ and the group $\Aut_{part}(\cN_{part})$ of automorphisms $T:\cN_{part}\ra \cN_{part}$ of the partial von Neumann algebra $\cN_{part}$.
\end{corollary}

To each von Neumann algebra $\cN$, we can associate a Jordan algebra, also denoted $\cN$: the Jordan algebra has the same elements and linear structure as the von Neumann algebra $\cN$ and is equipped with the \emph{Jordan product}
\begin{equation}
			\forall\hA,\hB\in\cN: \hA\cdot\hB=\frac{1}{2}(\hA\hB+\hB\hA).
\end{equation}
A unital Jordan $*$-morphism preserves the unit element, involution, linear structure and Jordan product. Jordan $*$-isomorphisms are necessarily unital.

Since a von Neumann algebra is weakly closed, the associated Jordan algebra will also be weakly closed, and hence is a \emph{$JBW$-algebra} (where the acronym stands for \emph{J}ordan-\emph{B}anach-\emph{W}eakly closed). For the theory of $JBW$-algebras, see \cite{AlfShu03} and references therein. There is a category $\JBW$ of complex, unital $JBW$-algebras and unital, normal Jordan $*$-morphisms. It is easy to check that each unital normal morphism $\phi:\cM\ra\cN$ of von Neumann algebras induces a unital normal Jordan $*$-morphism of the associated $JBW$-algebras, so there is a functor
\begin{align}
			\mc J: \vNa &\lra \JBW\\			\nonumber
			\cN &\lmt (\cN,\cdot)
\end{align}
sending each von Neumann algebra $\cN$ to the corresponding $JBW$-algebra $\cN$ with Jordan product $\cdot$, and each unital morphism $\phi:\cM\ra\cN$ to the corresponding morphism of $JBW$-algebras. The functor $\mc J$ is not injective on objects: a von Neumann algebra $\cN$ and its opposite algebra $\cN^o$, with the order of multiplication reversed, are mapped to the same $JBW$-algebra $(\cN,\cdot)$. More crucially, the functor $\mc J$ maps certain non-isomorphic von Neumann algebras to the same $JBW$-algebra : as Connes showed in \cite{Con75}, there is a factor $\cN$ that is not isomorphic to its opposite algebra $\cN^o$.

In the following, we will need a theorem by Dye \cite{Dye55}, see also \cite{Ham03}, Thm. 8.1.1:
\begin{theorem}			\label{Thm_Dye}
(Dye) Let $\cM,\,\cN$ be von Neumann algebras without type $I_2$ summands. Every isomorphism $\tT:\PM\ra\PN$ between the complete orthomodular lattices of projections induces a Jordan $*$-isomorphism $T:\cM\ra\cN$ such that $T(\hP)=\tT(\hP)$ for all $\hP\in\PM$. Conversely, every Jordan $*$-isomorphism $T:\cM\ra\cN$ restricts to an isomorphism $\tT:=T|_{\PM}:\PM\ra\PN$. (Clearly, the two maps $\tT\mt T$ and $T\mt\tT$ are inverse to each other.)
\end{theorem}

\begin{corollary}			\label{Cor_DyeOneAlg}
If $\cN$ is a von Neumann algebra without type $I_2$ summand, there is a group isomorphism
\begin{equation}
			C:\Aut_{cOML}(\PN) \lra \Aut_{Jordan}(\cN)
\end{equation}
between the group of automorphisms of the complete orthomodular lattice of projections in $\cN$ and the group of Jordan $*$-automorphisms of $\cN$, seen as a $JBW$-algebra.
\end{corollary}

Dye's result shows that if $\tT:\PM\ra\PN$ is an isomorphism of complete OMLs and $T:\cM\ra\cN$ is the associated Jordan $*$-automorphism, then $T|_{\PN}=\tT$ (that is, $T$ is an extension of $\tT$). The partial von Neumann algebra $\cM_{part}$ can be seen as `part' of the Jordan algebra $\cM$, and the restriction of the Jordan $*$-morphism $T$ to $\cM_{part}$ is the isomorphism of partial von Neumann algebras constructed in the proof of Prop. \ref{Prop_PartVNIsomsAreProjLattIsoms}: for this, note that on commuting operators $\hA,\,\hB\in\cM$, the Jordan product coincides with the product coming from $\cM$,
\begin{equation}
			\hA\hB=\hB\hA \quad\Longleftrightarrow\quad \hA\cdot\hB=\frac{1}{2}(\hA\hB+\hB\hA)=\hA\hB.
\end{equation}

For von Neumann algebras, there is a result connecting base maps $\ga:\VM\ra\VN$ between context categories and Jordan $*$-isomorphisms from $\cM$ to $\cN$. As was shown in \cite{HarDoe10},

\begin{theorem}			\label{Thm_OrderIsoGivesJordanAutom}
Let $\cM,\,\cN$ be von Neumann algebras not isomorphic to $\bbC^2=\bbC\oplus\bbC$ and without type $I_2$ summands, and let $\VM,\,\VN$ be their context categories, that is, the set of abelian von Neumann subalgebras of $\cM$ respectively $\cN$, each partially ordered under inclusion. For every order-isomorphism $\ga:\VM\ra\VN$, there exists a Jordan $*$-isomorphism $T:\cM\ra\cN$ such that $T[W]=\ga(W)$ for all $W\in\VM$. Conversely, every Jordan $*$-isomorphism $T:\cM\ra\cN$ induces an order-isomorphism $\ga:\VM\ra\VN$ such that $T[W]=\ga(W)$ for all $W\in\VM$.
\end{theorem}

Hence, two von Neumann algebras $\cM,\,\cN$ not isomorphic to $\bbC^2$ and without type $I_2$ summands are Jordan $*$-isomorphic if and only if their context categories $\VM,\,\VN$ are order-isomorphic.

The proof proceeds in two main steps: in the first step, using a result by Harding and Navara \cite{HarNav11}, one shows that every order-isomorphism (base map) $\ga:\VM\ra\VN$ determines a unique isomorphism $T:\PM\ra\PN$ of the projection lattices and vice versa. For this to apply, $\PM$ must not have $4$-element blocks (maximal Boolean sublattices), which is why we have to exclude the cases $\cM\sim\bbC^2$ and $\cM\sim\mc B(\bbC^2)$. (The latter is a type $I_2$ von Neumann algebra.) The second step is to apply the theorem by Dye (Thm. \ref{Thm_Dye}) cited above, which shows that $T:\PM\ra\PN$ extends to a Jordan $*$-automorphism $T:\cM\ra\cN$. Conversely, every Jordan $*$-automorphism $T:\cM\ra\cN$ determines a unique order-isomorphism $\ga:\VM\ra\VN$ of the context categories by $\ga(W):=T[W]$ for all $W\in\VM$. The two maps $\ga\mt T$ and $T\mt\ga$ are inverse to each other by construction.

\begin{corollary}			\label{Cor_OrderIsoGivesJordanAutom}
Let $\cN$ be a von Neumann algebra not isomorphic to $\bbC^2$ and without type $I_2$ summand, and let $\VN$ be its context category. There is a group isomorphism
\begin{equation}
			D:\Aut_{ord}(\VN) \lra \Aut_{Jordan}(\cN)
\end{equation}
between the group $\Aut_{ord}(\VN)$ of order-automorphisms	$\ga:\VN\ra\VN$ and the group $\Aut_{Jordan}(\cN)$ of Jordan $*$-automorphisms $T:\cN \lra \cN$.
\end{corollary}

Taken together, the results in this subsection imply:

\begin{theorem}			\label{Thm_VNAs}
Let $\cM,\,\cN$ be von Neumann algebras without type $I_2$ summands, with associated $JBW$-algebras also denoted $\cM,\,\cN$, spectral presheaves $\SigM$, $\SigN$, context categories $\VM,\,\VN$, partial von Neumann algebras $\cM_{part}$, $\cN_{part}$, and projection lattices $\PM,\,\PN$, respectively. Every isomorphism $\SigN\ra\SigM$ of the spectral presheaves induces a unique Jordan $*$-isomorphism $T:\cM\ra\cN$ (in the opposite direction). Conversely, every Jordan $*$-isomorphism $T:\cM\ra\cN$ induces a unique isomorphism $\pair{\Ga}{\io}:\SigN\ra\SigM$ of the spectral presheaves (in the opposite direction). Moreover, every Jordan $*$-isomorphism $T:\cM\ra\cN$ gives an isomorphism $\pair{\tilde\Ga}{\kappa}:\fM\ra\fN$ in the category of unital commutative $C^*$-algebra-valued presheaves from the Bohrification of $\cM$ to the Bohrification of $\cN$ and vice versa, and gives an isomorphism $T:\cM_{part}\ra\cN_{part}$ of partial von Neumann algebras and vice versa, which in turn gives an isomorphism $\tT:\PM\ra\PN$ of complete orthomodular lattices and vice versa, which, if $\cM\nsim\bbC^2$, induces an order-isomorphism $\ga:\VM\ra\VN$ and vice versa. Here, $\ga$ is the base map underlying the essential geometric morphism $\pair{\Ga}{\io}:\SigN\ra\SigM$ corresponding to the Jordan $*$-isomorphism $T:\cM\ra\cN$.
\end{theorem}

\begin{proof}
Prop. \ref{Prop_IsomsOfSpecPreshsAreJordanIsoms} (which actually uses Prop. \ref{Prop_IsomsOfSpecPreshsGiveIsomsAsPartialAlgs}, see next section) establishes a bijective correspondence between isomorphisms $\pair{\Ga}{\io}:\SigN\ra\SigM$ of spectral presheaves and isomorphisms $T:\cM_{part}\ra\cN_{part}$ of partial von Neumann algebras. Prop. \ref{Prop_PartVNIsomsAreProjLattIsoms} gives a bijective correspondence between isomorphisms $T:\cM_{part}\ra\cN_{part}$ of partial von Neumann algebras and isomorphisms $\tT:\PM\ra\PN$ of their projection lattices, and the theorem by Dye, Thm. \ref{Thm_Dye}, gives the bijective correspondence between lattice isomorphisms $\tT:\PM\ra\PN$ and Jordan $*$-isomorphisms $T:\cM\ra\cN$. Hence, every isomorphism $\pair{\Ga}{\io}:\SigN\ra\SigM$ determines a unique Jordan $*$-isomorphism $T:\cM\ra\cN$ and vice versa. Moreover, Thm. \ref{Thm_OrderIsoGivesJordanAutom} shows that if $\cM\nsim\bbC^2$, there is a bijective correspondence between isomorphisms $\tT:\PM\ra\PN$ of the projection lattices and order-isomorphisms $\ga:\VM\ra\VN$ of the context categories. By construction, $\ga$ is the base map underlying the essential geometric morphism $\Ga$ in the isomorphism $\pair{\Ga}{\io}:\SigN\ra\SigM$.
\end{proof}

This shows that a von Neumann algebra $\cN$ with no type $I_2$ summand is determined up to Jordan $*$-isomorphism by its spectral presheaf $\SigN$. Moreover, the spectral presheaf is `rigid' in the sense that if $\cM\nsim\bbC^2$, every isomorphism $\pair{\Ga}{\io}:\SigN\ra\SigM$ corresponds to a unique base map $\ga:\VM\ra\VN$ and vice versa.

As a corollary of Thm. \ref{Thm_VNAs}, we obtain:
\begin{corollary}			\label{Cor_GroupIsos}
Let $\cN$ be a von Neumann algebra not isomorphic to $\bbC^2$ and without summand of type $I_2$, with projection lattice $\PN$, context category $\VN$, associated partial von Neumann algebra $\cN_{part}$, associated Jordan algebra also denoted $\cN$, and spectral presheaf $\SigN$.

The four groups $\Aut_{ord}(\VN)$, $\Aut_{cOML}(\PN)$, $\Aut_{part}(\cN_{part})$ and $\Aut_{Jordan}(\cN)$ are isomorphic. Concretely, every order-isomorphism (base map) $\tT\in\Aut_{ord}(\VN)$ induces a unique automorphism $T:\PN\ra\PN$ of the complete orthomodular lattice of projections, which extends to a partial von Neumann automorphism $T:\cN_{part}\ra \cN_{part}$ and further to a Jordan $*$-automorphism $T:\cN\ra\cN$. Conversely, each Jordan $*$-automorphism restricts to an automorphism of the partial algebra $\cN_{part}$, further to an automorphism of $\PN$, and induces an order-automorphism of $\VN$.

Each base map $\tT:\VN\ra\VN$ induces an automorphism $\pair{\tT}{T^*}:\SigN\ra\SigN$ of the spectral presheaf, and the group $\Aut(\SigN)$ of automorphisms of the spectral presheaf is contravariantly isomorphic to the groups $\Aut_{ord}(\VN)$, $\Aut_{cOML}(\PN)$, $\Aut_{part}(\cN_{part})$ and $Aut_{Jordan}(\cN)$.
\end{corollary}

This formulation will be useful when considering time evolution and flows on the spectral presheaf in \cite{Doe12c}.

\subsection{Unital $C^*$-algebras}			\label{Subsec_CStarAlgs}
\begin{definition}			\label{Def_IsomsFromSigAToSigB}
Let $\cA,\,\cB$ be unital $C^*$-algebras, and let $\SigA,\,\SigB$ be their spectral presheaves. An \emph{isomorphism from $\SigB$ to $\SigA$} is a pair $\pair{\Ga}{\io}$, where $\Ga:\SetCAop\ra\SetCBop$ is an essential geometric isomorphism, induced by an order-isomorphism $\ga:\CA\ra\CB$ (called the \emph{base map}). $\Ga^*:\SetCBop\ra\SetCAop$ is the inverse image functor of the geometric isomorphism $\Ga$, and $\io:\Ga^*(\SigB)\ra\SigA$ is a natural isomorphism for which each component $\io_C:(\Ga^*(\SigB))_C\ra\SigA_C$, where $C\in\CA$, is a homeomorphism. Hence, an isomorphism $\pair{\Ga}{\io}$ acts by
\begin{equation}
			\SigB \stackrel{\Ga^*}{\lra} \Ga^*(\SigB) \stackrel{\io}{\lra} \SigA.
\end{equation}
We will also use the notation $\io\circ\Ga^*$ for an isomorphism $\pair{\Ga}{\io}$. If $\cA=\cB$, an isomorphism $\pair{\Ga}{\io}:\SigA\ra\SigA$ is called an \emph{automorphism of $\SigA$}. \end{definition}

It is easy to see that the automorphisms of $\SigA$ form a group, which we denote as $\Aut(\SigA)$. The group operation is
\begin{align}
			\Aut(\SigA)\times\Aut(\SigA) &\lra \Aut(\SigA)\\			\nonumber
			(\pair{\Ga_1}{\io_1},\pair{\Ga_2}{\io_2}) &\lmt \pair{\Ga_2\circ\Ga_1}{\io_1\circ\io_2},
\end{align}
where $\Ga_2\circ\Ga_1:\SetVNop\ra\SetVNop$ is the essential geometric automorphism induced by $\ga_2\circ\ga_1$, the composite of the base maps underlying $\Ga_2$ respectively $\Ga_1$, and $\io_1\circ\io_2$ is the natural isomorphism with components
\begin{equation}
			\forall V\in\VN: (\io_1\circ\io_2)_V=\io_{1;V}\circ\io_{2;\ga_1(V)}
\end{equation}
(cf. Def. \ref{Def_CatOfPresheaves} and Lemma \ref{Lem_SigNIsGroup}). An isomorphism in the sense of Def. \ref{Def_IsomsFromSigAToSigB} from $\SigB$ as an object in the topos $\SetCBop$ to $\SigA$ as an object in the topos $\SetCAop$ corresponds to an isomorphism from $\SigB$ to $\SigA$ as objects in the category $\Pre{\KHaus}$, see Def. \ref{Def_CatOfPresheaves}. Conversely, each isomorphism from $\SigB$ to $\SigA$ in $\Pre{\KHaus}$ determines a unique isomorphism from $\SigB$ as an object of $\SetCBop$ to $\SigA$ as an object of $\SetCAop$.

As in the case of von Neumann algebras, requiring that the base map $\ga:\CA\ra\CB$ is an order-isomorphism is equivalent to saying that $\ga$ is an invertible covariant functor. More succinctly, we could just require that there is an essential geometric isomorphism $\Ga:\SetCAop\ra\newline\SetCBop$ (which implies that there is an underlying invertible covariant functor $\ga:\CA\ra\newline\CB$). 

We will show in Prop. \ref{Prop_RepOfAutcAInAutSigA} that every isomorphism from a unital $C^*$-algebra $\cA$ to a unital $C^*$-algebra $\cB$ gives an isomorphism from $\SigB$ to $\SigA$ in the sense defined above.

Using a result by Hamhalter \cite{Ham11}, we will show that if $\cA,\,\cB$ are neither isomorphic to $\bbC^2$ nor to $\mc B(\bbC^2)$, then every base map $\ga$ induces a unique natural isomorphism $\io$ and hence a unique isomorphism $\pair{\Ga}{\io}:\SigB\ra\SigA$ of the spectral presheaf, see Thm. \ref{Thm_SpecPresheafDeterminesPartialUnitalCStarAlg}. 

We now characterise isomorphisms from $\SigB$ to $\SigA$:

\begin{lemma}			\label{Lem_NatIsoGivesPartialAutom}
Let $\cA,\,\cB$ be unital $C^*$-algebras, with spectral presheaves $\SigA,\,\SigB$. Each isomorphism $\pair{\Ga}{\io}:\SigB\ra\SigA$ induces an isomorphism $T:\cA_{part}\ra \cB_{part}$ from the partial unital $C^*$-algebra $\cA_{part}$ of normal operators in $\cA$ to the partial unital $C^*$-algebra $\cB_{part}$ of normal operators in $\cB$.
\end{lemma}

\begin{proof}
Since $\pair{\Ga}{\io}:\SigB\ra\SigA$ is an isomorphism, the base map $\ga:\CA\ra\CB$ corresponding to the essential geometric isomorphism $\Ga:\SetCAop\ra\SetCBop$ is an order-isomorphism, and each component 
\begin{equation}
			\io_{C'}:\Ga^*(\SigB)_{C'}=\SigB_{\ga(C')}\lra\SigA_{C'},
\end{equation}
where $C'\in\CA$, is a homeomorphism. By Gelfand duality, this corresponds to a unique (unital) isomorphism
\begin{align}
			k_{C'}: C(\SigA_{C'}) &\lra C(\SigB_{\ga(C')})\\			\nonumber
			f &\lmt f\circ\io_{C'}.
\end{align}
of abelian $C^*$-algebras. Since Gelfand duality (see \eq{Eq_GelfandDuality}) is a dual equivalence, there is a natural isomorphism $\eta:\id_{\ucC}\ra C(-)\circ\Sig$, the unit of the adjunction. Thus, there are in particular isomorphisms $\eta_{C'}: C'\ra C(\SigA_{C'})$ and $\eta_{\ga(C')}^{-1}:C(\SigB_{\ga(C')})\ra\ga(C')$ of unital abelian $C^*$-algebras, so we get an isomorphism
\begin{equation}
			\kappa_{C'}=\eta_{\ga(C')}^{-1}\circ k_{C'}\circ\eta_{C'}: C' \lra \ga(C')
\end{equation}
between the abelian $C^*$-algebras ${C'}$ and $\ga(C')$, for every $C'\in\CA$. We note that $C'=\fA_{C'}$ is the component at $C'\in\CA$ of the Bohrification of $\cA$ (see Def. \ref{Def_Bohrification}). Analogously, $\ga(C')=\fB_{\ga(C')}=(\tilde\Ga^*(\fB))_{C'}$ is the component at $C'$ of $\tilde\Ga^*(\fB)$, the pullback of the Bohrification of $\cB$ by the essential geometric morphism $\tilde\Ga:\SetCA\ra\SetCB$ of copresheaf topoi that is induced by the base map $\ga:\CA\ra\CB$ that is underlying the given isomorphism $\pair{\Ga}{\io}:\SigB\ra\SigA$ of the spectral presheaves.

Let $\kappa=(\kappa_C)_{C\in\CA}$. Naturality of $\io$ and $\eta$ easily imply that
\begin{align}
	\kappa: \fA \lra \tilde\Ga^*(\fB)
\end{align}
is a natural transformation, and the fact that $\io$ has an inverse $\io^{-1}:\SigA\ra\Ga^*(\SigB)$ implies that $\kappa$ has an inverse $\kappa^{-1}:\tilde\Ga^*(\fB)\ra\fA$, too. Hence, $\kappa:\fA\ra\tilde\Ga^*(\fB)$ is a natural isomorphism in the topos $\SetCA$ for which each component $\kappa_C$, $C\in\CA$, is an isomorphism of abelian $C^*$-algebras. Define a map
\begin{align}
			T: \cA_{part} &\lra \cB_{part}\\			\nonumber
			\hA &\lmt \kappa_C(\hA)
\end{align}
from the set of normal operators in $\cA$ to the normal operators in $\cB$, where $C\in\CA$ is an abelian subalgebra that contains $\hA$. This map is well-defined (that is, the value $T(\hA)=\kappa_C(\hA)$ does not depend on which algebra $C$ containing $\hA$ we choose), since $\kappa=(\kappa_C)_{C\in\CA}$ is natural, and the copresheaf maps of $\fA$ are simply inclusions. Clearly, $T(\hat 1)=\hat 1$.
\end{proof}

Conversely, we have 

\begin{lemma}			\label{Lem_PartialMorGivesNatIso}
Every isomorphism $T:\cA_{part}\ra \cB_{part}$ from a unital partial $C^*$-algebra $\cA_{part}$ to a unital partial $C^*$-algebra $\cB_{part}$ induces an isomorphism $\pair{\Ga}{\io}:\SigB\ra\SigA$, that is, an order-isomorphism $\ga:\CA\ra\CB$ with corresponding essential geometric isomorphism $\Ga:\SetCAop\ra\SetCBop$ and a natural isomorphism $\io:\Ga^*(\SigB)\ra\SigA$ such that each component $\io_C$ is a homeomorphism.
\end{lemma}

\begin{proof}
Let $C\in\CA$ be a unital abelian $C^*$-subalgebra of $\cA$. Then $T|_C:C\ra \cB_{part}$ is a unital $*$-homomorphism from the abelian $C^*$-algebra $C$ into $\cB_{part}$, so $T|_C(C)$ is norm-closed and hence a unital abelian $C^*$-subalgebra of $\cB_{part}\subset\cB$, that is, $T|_C(C)\in\CB$.

Clearly, $C'\subset C$ implies $T|_{C'}(C')\subset T|_C(C)$, so we have a monotone map
\begin{align}
			\ga:\CA &\lra \CB\\			\nonumber
			C &\lmt T|_C(C).
\end{align}
Moreover, since $T$ has an inverse $T^{-1}$ that is a unital partial $*$-isomorphism as well, the map $\ga$ has an inverse $\ga^{-1}:\CB\ra\CA$, too, and hence $\ga$ is an order-isomorphism (base map).

Since $T$ is an isomorphism of partial unital $C^*$-algebras, $T|_C:C\ra T(C)$ is isomorphism of unital abelian $C^*$-algebras. There is a natural isomorphism $t:\fA\ra\Ga^*(\fB)$ in $\SetCA$, with components
\begin{equation}
			\forall C\in\CA: t_C:=T|_C: C \lra T(C).
\end{equation}
Define, for each $C\in\CA$,
\begin{align}
			\io_C: (\Ga^*(\SigB))_C = \SigB_{\ga(C)} &\lra \SigA_C\\			\nonumber
			\ld &\lmt \ld\circ T|_C
\end{align}
the homeomorphism between the Gelfand spectra of $\ga(C)$ and $C$ corresponding to $T|_C$. By construction, the $\io_C$ are the components of a natural isomorphism $\io:\Ga^*(\SigB)\ra\SigA$.
\end{proof}

Clearly, the constructions in Lemma \ref{Lem_NatIsoGivesPartialAutom} and Lemma \ref{Lem_PartialMorGivesNatIso} are inverse to each other. Summing up, we have shown:

\begin{proposition}			\label{Prop_IsomsOfSpecPreshsGiveIsomsAsPartialAlgs}
Let $\cA,\,\cB$ be unital $C^*$-algebras. There is a bijective correspondence between isomorphisms $\pair{\Ga}{\io}:\SigB\ra\SigA$ of the spectral presheaves and isomorphisms $T:\cA_{part}\ra\cB_{part}$ of the associated partial unital $C^*$-algebras. Hence, $\cA$ and $\cB$ are isomorphic as partial unital $C^*$-algebras if and only if their spectral presheaves are isomorphic.
\end{proposition}

For a moment, we assume $\cA=\cB$ and consider automorphisms of the spectral presheaf $\SigA$. If $\pair{\Ga_1}{\io_1},\pair{\Ga_2}{\io_2}$ are two automorphisms of $\SigA$, and $T_1,T_2$ are the corresponding partial $*$-automorphisms of $\cA_{part}$, then by construction, the partial $*$-automorphism corresponding to the composite automorphism $\pair{\Ga_1}{\io_1}\circ\pair{\Ga_2}{\io_2}$ is $T_2\circ T_1$.

\begin{corollary}			\label{Cor_ContravGroupIso}
There is a contravariant group isomorphism between $\Aut(\SigA)$, the group of automorphisms of $\SigA$, and $\Aut_{part}(\cA_{part})$, the group of unital partial $*$-automorphisms of $\cA_{part}$.
\end{corollary}

\begin{definition}
Let $\cA,\,\cB$ be unital $C^*$-algebras, and let $\fA,\,\fB$ be their Bohrifications (see Def. \ref{Def_Bohrification}). An \emph{isomorphism from $\fA$ to $\fB$} is a pair $\pair{\tilde\Ga}{\kappa}$, where $\tilde\Ga:\SetCA\ra\SetCB$ is an essential geometric isomorphism, induced by an order-isomorphism $\ga:\CA\ra\CB$ (the \emph{base map}). $\tilde\Ga^*:\SetCB\ra\SetCA$ is the inverse image functor of the geometric isomorphism $\tilde\Ga$, and $\kappa:\fA\ra\tilde\Ga^*(\fB)$ is a natural isomorphism for which each component $\kappa_C:\fA_C\ra\tilde\Ga^*(\fB)_C$, $C\in\CA$, is a unital $*$-isomorphism (of abelian $C^*$-algebras). If $\cA=\cB$, an isomorphism $\pair{\tilde\Ga}{\kappa}:\fA\ra\fA$ is called an \emph{automorphism of $\fA$}.
\end{definition}

The automorphisms of $\fA$ form a group, which we denote as $\Aut(\fA)$.

An isomorphism from $\fA$ as an object of the topos $\SetCA$ to $\ol\cB$ as an object of the topos $\SetCB $ in the sense defined above corresponds to an isomorphism from $\fA$ to $\fB$ in the category $\Cpr{\ucC}$ of copresheaves with values in unital abelian $C^*$-algebras (see Def. \ref{Def_CatOfCopresheaves}). Prop. \ref{Prop_DualEquiv} implies:

\begin{corollary}			\label{Cor_AutSigAContravIsomToAutolcA}
Let $\cA,\,\cB$ be unital $C^*$-algebras. Every isomorphism $\pair{\tilde\Ga}{\kappa}:\fA\ra\fB$ between their Bohrifications corresponds to an isomorphism $\pair{\Ga}{\io}:\SigB\ra\SigA$ in the opposite direction between their spectral presheaves. If $\cA=\cB$, there is a contravariant group isomorphism
\begin{equation}
			\Aut(\SigA) \lra \Aut(\fA)
\end{equation}
between the group of automorphisms of the spectral presheaf of $\cA$ and the group of automorphisms of the Bohrification of $\cA$.
\end{corollary}

Noting that $\cA_{part}$ and $\cB_{part}$ are the `external' descriptions of the topos-internal abelian algebras $\fA\in\SetCAop$ respectively $\fB\in\SetCBop$ (cf. Rem. \ref{Rem_fAIsPartialAlg}), one can read the proofs of Lemmas \ref{Lem_NatIsoGivesPartialAutom} and \ref{Lem_PartialMorGivesNatIso} as making this correspondence between isomorphisms from $\SigB$ to $\SigA$ and isomorphisms from $\fA$ to $\fB$ explicit.

The following result links automorphisms of a unital $C^*$-algebra $\cA$ and automorphisms of its spectral presheaf $\SigA$:

\begin{proposition}			\label{Prop_RepOfAutcAInAutSigA}
Let $\cA$ be a unital $C^*$-algebra, and let $\SigA$ be its spectral presheaf. There is an injective group homomorphism
\begin{align}
			\Aut(\cA) &\lra \Aut(\SigA)^{\op}\\			\nonumber
			\phi &\lmt \pair{\Phi}{\cG_\phi}=\cG_{\phi}\circ\Phi^*
\end{align}
from the automorphism group of $\cA$ into the opposite group of the automorphism group of $\SigA$. (This is the same as an injective, contravariant group homomorphism from $\Aut(\cN)$ into $\Aut(\SigN)$.) Here, $\pair{\Phi}{\cG_\phi}$ is the automorphism of $\SigA$ induced by $\phi$ as in Section \ref{Sec_MorsAndMaps} (see also Thm. \ref{Thm_TwoFunctors}).
\end{proposition}

\begin{proof}
The result is a corollary of what we proved so far. It follows from Prop. \ref{Prop_DualEquiv} and Thm. \ref{Thm_TwoFunctors} (though injectivity would need a separate argument). More explicitly, one can argue: let $\phi:\cA\ra\cA$ be an automorphism. By restricting $\phi$ to the partial $C^*$-algebra $\cA_{part}$, we obtain an automorphism $\phi|_{\cA_{part}}$ of $\cA_{part}$, and by Cor. \ref{Cor_ContravGroupIso}, this corresponds to an automorphism of $\SigA$. By the fact that the groups $\Aut_{part}(\cA_{part})$ and $\Aut(\SigA)$ are contravariantly isomorphic, this gives an injective, contravariant group homomorphism from $\Aut(\cA)$ into $\Aut(\SigA)$ (which is the same as a group homomorphism from $\Aut(\cA)$ into $\Aut(\SigA)^{\op}$).

If $\phi,\,\xi:\cA\ra\cA$ are two distinct automorphisms of $\cA$, then there exists some self-adjoint operator $\hA\in \cA_{part}$ for which $\phi(\hA)\neq\xi(\hA)$, so the two automorphisms $\phi|_{\cA_{part}}$ and $\xi|_{\cA_{part}}$ of the partial $C^*$-algebra $\cA_{part}$ are distinct, too, so the group homomorphism from $\Aut(\cA)$ into $\Aut(\SigA)^{\op}$ is injective.
\end{proof}

Yet, $\Aut(\SigA)$ has more elements than those corresponding to elements of $\Aut(\cA)$. Each $\phi\in\Aut(\SigA)$ induces a unital partial $*$-automorphism $T\in\Aut_{part}(\cA_{part})$ of the partial $C^*$-algebra $\cA_{part}$ of normal elements, but not every unital partial $*$-automorphism $T$ induces a $C^*$-automorphism $\phi\in\Aut(\cA)$.


Since a $C^*$-algebra is norm-closed, the associated Jordan algebra will also be norm-closed, and hence is a \emph{$JB$-algebra} (where the acronym stands for \emph{J}ordan-\emph{B}anach). For the theory of $JB$-algebras, see \cite{AlfShu03} and references therein. There is a category $\JB$ of complex, unital $JB$-algebras and unital Jordan $*$-morphisms. It is easy to check that each unital normal morphism $\phi:\cA\ra\cB$ of von Neumann algebras induces a unital normal Jordan $*$-morphism of the associated $JB$-algebras, so there is a functor
\begin{align}
			\mc J: \ucC &\lra \JB\\			\nonumber
			\cA &\lmt (\cA,\cdot).
\end{align}
Just as for von Neumann algebras and $JBW$-algebras, the functor $\mc J$ is not injective on objects: a unital $C^*$-algebra $\cA$ and its opposite algebra $\cA^o$, with the order of multiplication reversed, are mapped to the same $JB$-algebra $(\cA,\cdot)$.

In the arguments above, we started from an isomorphism $\pair{\Gamma}{\io}:\SigB\ra\SigA$, where the geometric automorphism $\Ga:\SetCAop\ra\SetCBop$ is induced by a base map $\ga:\CA\ra\CB$ and the natural isomorphism $\io:\Ga^*(\SigB)\ra\SigA$ is given as additional data. 

It is also interesting to start just from an order-automorphism (base map) $\ga:\CA\ra\CB$, without assuming that a natural isomorphism $\io:\Ga^*(\SigB)\ra\SigA$ is given. This amounts to asking how much operator-algebraic structure can be reconstructed from the poset $\CA$ of unital abelian $C^*$-subalgebras of a nonabelian $C^*$-subalgebra $\cA$. (For the case of von Neumann algebras, see Thm. \ref{Thm_OrderIsoGivesJordanAutom}.)

There is an interesting recent result by Jan Hamhalter \cite{Ham11} on reconstructing parts of the algebraic structure of a unital $C^*$-algebra $\cA$ from its poset $\CA$ of unital abelian subalgebras. We need a definition first: let $\cA,\,\cB$ be unital $C^*$-algebras, and let $\cA_{\sa}$ be the real unital Jordan algebra of self-adjoint elements in $\cA$, with Jordan product given by
\begin{equation}
			\forall\hA,\hB\in\cA_{\sa}: \hA\cdot\hB:=\frac{1}{2}(\hA\hB+\hB\hA).
\end{equation}
The unital Jordan algebra $\cB_{\sa}$ is defined analogously. Hamhalter defines a \emph{quasi-Jordan homomorphism} to be a unital map
\begin{equation}
			Q:\cA_{\sa} \lra \cB_{\sa}
\end{equation}
such that, for all $C\in\CA$,
\begin{equation}
			Q|_{C_{\sa}}: C_{\sa} \lra \cB_{\sa}
\end{equation}
is a unital Jordan homomorphism. Note that $Q$ is only required to be linear on commuting self-adjoint operators, so it is a quasi-linear map. Moreover, $Q$ preserves the Jordan product on commuting operators (where the Jordan product coincides with the operator product since
\begin{equation}
			\hA\hB=\hB\hA \quad\Longleftrightarrow\quad \hA\cdot\hB=\frac{1}{2}(\hA\hB+\hB\hA)=\hA\hB,
\end{equation}
as we had already observed for the case of von Neumann algebras).

A \emph{quasi-Jordan isomorphism} is a bijective map $Q:\cA_{\sa}\ra\cB_{\sa}$ such that $Q$ and $Q^{-1}$ are quasi-Jordan homomorphisms.

\begin{theorem}			\label{Thm_Hamhalter}
(Hamhalter \cite{Ham11}) Let $\cA,\,\cB$ be unital $C^*$-algebras such that $\cA$ is neither isomorphic to $\bbC^2$ nor to $\mc B(\bbC^2)$. There is an order-isomorphism
\begin{equation}
			\ga:\CA \lra \CB
\end{equation}
if and only if there is a unital quasi-Jordan isomorphism
\begin{equation}
			Q:\cA_{\sa} \lra \cB_{\sa}.
\end{equation}
\end{theorem}

The unital quasi-Jordan algebra $\cA_{\sa}$ is closely related to our unital partial $C^*$-algebra $\cA_{part}$, it is simply the self-adjoint part of it. Moreover, each isomorphism $T:\cA_{part}\ra\cB_{part}$ obviously restricts to a unital quasi-Jordan isomorphism $T|_{\cA_{\sa}}:\cA_{\sa}\ra\cA_{\sa}$. Conversely, every unital quasi-Jordan automorphism $Q:\cA_{\sa}\ra\cB_{\sa}$ extends to an isomorphism $Q:\cA_{part}\ra\cB_{part}$ of partial unital $C^*$-algebras by linearity (on commuting operators). If $\cA=\cB$, there is a group isomorphism
\begin{equation}
			\Aut_{quJord}(\cA_{\sa}) \lra \Aut_{part}(\cA_{part}),
\end{equation}
where $\Aut_{quJord}(\cA_{sa})$ is the group of quasi-Jordan automorphisms of $\cA_{\sa}$. 

We reformulate Thm. \ref{Thm_Hamhalter} as

\begin{theorem}			\label{Thm_HamhalterVersion}
Let $\cA,\cB$ be unital $C^*$-algebras that are neither isomorphic to $\bbC^2$ nor to $\mc B(\bbC^2)$. There are bijective correspondences between the set of order-isomorphisms $\ga:\CA\ra\CB$, the set of unital quasi-Jordan isomorphisms $Q:\cA_{\sa}\ra\cB_{\sa}$, and the set of partial unital $*$-isomorphisms $T:\cA_{part}\ra\cB_{part}$.
\end{theorem}

\begin{corollary}			\label{Cor_HamhalterVersion}
Let $\cA$ be a unital $C^*$-algebra that is neither isomorphic to $\bbC^2$ nor to $\mc B(\bbC^2)$. There are group isomorphisms
\begin{equation}
			\Aut_{ord}(\CA) \lra \Aut_{quJord}(\cA_{\sa}) \lra \Aut_{part}(\cA_{part}).
\end{equation}
\end{corollary}

This shows that if $\cA$ is neither isomorphic to $\bbC^2$ nor to $\mc B(\bbC^2)$, then every base map $\ga:\CA\ra\CA$ induces a unique automorphism of the unital partial $C^*$-algebra $\cA_{part}$. From Prop. \ref{Prop_IsomsOfSpecPreshsGiveIsomsAsPartialAlgs}, Thm. \ref{Thm_HamhalterVersion} and the other results in this subsection, we obtain

\begin{theorem}			\label{Thm_SpecPresheafDeterminesPartialUnitalCStarAlg}
Let $\cA,\,\cB$ be unital $C^*$-algebras which are neither isomorphic to $\bbC^2$ nor to $\mc B(\bbC^2)$, with associated $JB$-algebras also denoted $\cA,\,\cB$, spectral presheaves $\SigA$, $\SigB$, context categories $\CA,\,\CB$, and partial unital $C^*$-algebras $\cA_{part}$, $\cB_{part}$, respectively. Every isomorphism $\SigB\ra\SigA$ of the spectral presheaves induces a unique quasi-Jordan $*$-isomorphism $T:\cA\ra\cB$ (in the opposite direction). Conversely, every quasi-Jordan $*$-isomorphism $T:\cA\ra\cB$ induces a unique isomorphism $\pair{\Ga}{\io}:\SigB\ra\SigA$ of the spectral presheaves (in the opposite direction). Moreover, every quasi-Jordan $*$-isomorphism $T:\cA\ra\cB$ gives an isomorphism $\pair{\tilde\Ga}{\kappa}:\fA\ra\fB$ in the category $\Cpr{\ucC}$ of unital commutative $C^*$-algebra-valued presheaves from the Bohrification of $\cA$ to the Bohrification of $\cB$ and vice versa, which induces an order-isomorphism $\ga:\CA\ra\CB$ of the context categories and vice versa. Here, $\ga$ is the base map underlying the essential geometric morphism $\pair{\Ga}{\io}:\SigB\ra\SigA$ corresponding to the quasi-Jordan $*$-isomorphism $T:\cA\ra\cB$.
\end{theorem}

Note that this result is weaker than the one for von Neumann algebras (Thm. \ref{Thm_VNAs}), since isomorphisms of the spectral presheaves only determine quasi-Jordan $*$-isomorphisms, or equivalently, isomorphisms of unital partial $C^*$-algebras.

But, crucially, also for a unital $C^*$-algebra $\cA$ (not isomorphic to $\bbC^2$ or $\mc B(\bbC^2)$), every base map $\ga:\CA\ra\CA$ induces a unique automorphism of the spectral presheaf $\SigA$ of $\cA$, so the spectral presheaf is `rigid' in the sense that it has exactly as many automorphisms as the underlying base category $\CA$.

\begin{corollary}			\label{Cor_GroupIsosCStar}
Let $\cA$ be a unital $C^*$-algebra that is neither isomorphic to $\bbC^2$ nor to $\mc B(\bbC^2)$. The four groups $\Aut_{ord}(\CA)$, $\Aut_{quJord}(\cA_{\sa})$, $\Aut_{part}(\cA_{part})$ and $\Aut{\fA}$ are isomorphic, and these groups are contravariantly isomorphic to the group $\Aut(\SigA)$ of automorphisms of the spectral presheaf $\SigA$ of $\cA$.
\end{corollary}

For a certain class of unital $C^*$-algebras (which is strictly larger than the class of von Neumann algebras not isomorphic to $\bbC^2$ and without type $I_2$ summand), a stronger result can be obtained and the spectral presheaf $\SigA$ can be shown to determine the algebra $\cA$ up to Jordan $*$-isomorphisms. 

By a standard argument going back to Kadison \cite{Kad51}, a \emph{linear} unital quasi-Jordan automorphism $T:\cA_{sa}\ra\cA_{\sa}$ is in fact a Jordan automorphism, that is, it preserves the Jordan product also on noncommuting operators $\hA,\,\hB\in\cA_{\sa}$.\footnote{We note in passing that this implies that every partial unital $C^*$-algebra (respectively partial von Neumann algebra) for which addition is defined globally and not just between commuting elements is a unital Jordan algebra, and more precisely a $JB$-algebra (respectively $JBW$-algebra).} The reason is that for all $\hA,\,\hB\in\cA_{\sa}$, the Jordan product can be written as
\begin{equation}
			\hA\cdot\hB = \frac{1}{2}((\hA+\hB)^2-\hA^2-\hB^2).
\end{equation}
Then, using linearity of $T$, one easily sees that $T(\hA\cdot\hB)=T(\hA)\cdot T(\hB)$.

Following \cite{Ham11}, the question hence is: which unital quasi-linear Jordan automorphisms $\cA_{\sa}\ra\cA_{\sa}$ are linear? The problem which quasi-linear maps from an operator algebra into the complex numbers (or, more generally, into a Banach space) are linear has attracted substantial efforts by many authors and eventually led to the following deep result by Bunce and Wright:

\begin{theorem}
(Generalised Gleason Theorem, Bunce and Wright \cite{BunWri95}) Let $\cN$ be a von Neumann algebra with no type $I_2$ summand, and let $X$ be a Banach space. Let $\cA$ be a $C^*$-algebra that is a quotient of a norm-closed two-sided ideal $I$ in $\cN$. Suppose that $T:\cN_{\sa}\ra X$ is a quasi-linear map that is bounded on the unit ball. Then $T$ is linear.
\end{theorem}

As a corollary, Hamhalter obtains

\begin{corollary}			\label{Cor_HamCStarJordan}
(Cor. 3.6, \cite{Ham11}) Let $\cN$ be a von Neumann algebra with no type $I_2$ summand, and let $\cA$ be a $C^*$-algebra that is an at least three-dimensional quotient of an ideal algebra $\hat 1+I$, where $I$ is a norm-closed two-sided ideal in $\cN$. Let $\cB$ be a $C^*$-algebra. For each order-isomorphism $\ga:\CA\ra\CB$, there is a unique Jordan isomorphism $T:\cA_{\sa}\ra\cB_{\sa}$ that induces $\ga$.
\end{corollary}

We see that for the class of $C^*$-algebras described in Cor. \ref{Cor_HamCStarJordan} (which clearly contains all von Neumann algebras not isomorphic to $\bbC^2$ and with no type $I_2$ summand), the group $\Aut_{quJord}(\cA_{\sa})$ of unital quasi-Jordan automorphisms in fact is the group $\Aut_{Jordan}(\cA_{\sa})$. Each Jordan automorphism $T\in\Aut_{Jordan}(\cA_{\sa})$ can be extended by linearity to a Jordan $*$-automorphism of $\cA$. Conversely, each Jordan $*$-automorphism of $\cA$ restricts to a Jordan automorphism of $\cA_{\sa}$, and the two maps are inverse to each other. So, the groups $\Aut_{Jordan}(\cA_{\sa})$ and $\Aut_{Jordan}(\cA)$ are isomorphic.

Hence, for the particular class of unital $C^*$-algebras described in Cor. \ref{Cor_HamCStarJordan}, the spectral presheaf determines the Jordan structure completely, and we obtain

\begin{theorem}			\label{Thm_SpecPresheafOftenDeterminesUnitalCStarAlg}
Let $\cN$ be a von Neumann algebra with no type $I_2$ summand, and let $\cA$ be a $C^*$-algebra that is an at least three-dimensional quotient of an ideal algebra $\hat 1+I$, where $I$ is a norm-closed two-sided ideal in $\cN$. Let $\cB$ be a $C^*$-algebra. Every isomorphism $\SigB\ra\SigA$ of the spectral presheaves induces a unique Jordan $*$-isomorphism $T:\cA\ra\cB$ (in the opposite direction). Conversely, every Jordan $*$-isomorphism $T:\cA\ra\cB$ induces a unique isomorphism $\pair{\Ga}{\io}:\SigB\ra\SigA$ of the spectral presheaves (in the opposite direction). Moreover, every quasi-Jordan $*$-isomorphism $T:\cA\ra\cB$ gives an isomorphism $\pair{\tilde\Ga}{\kappa}:\fA\ra\fB$ in the category $\Cpr{\ucC}$ of unital commutative $C^*$-algebra-valued presheaves from the Bohrification of $\cA$ to the Bohrification of $\cB$ and vice versa, which induces an order-isomorphism $\ga:\CA\ra\CB$ of the context categories and vice versa. Here, $\ga$ is the base map underlying the essential geometric morphism $\pair{\Ga}{\io}:\SigB\ra\SigA$ corresponding to the quasi-Jordan $*$-isomorphism $T:\cA\ra\cB$.
\end{theorem}

\begin{corollary}			\label{Cor_GroupIsosCStar2}
Let $\cN$ be a von Neumann algebra with no type $I_2$ summand, and let $\cA$ be a $C^*$-algebra that is an at least three-dimensional quotient of an ideal algebra $\hat 1+I$, where $I$ is a norm-closed two-sided ideal in $\cN$. The groups $\Aut_{ord}(\CA)$, $\Aut_{Jordan}(\cA_{\sa})$, $\Aut_{part}(\cA_{part})$, $\Aut_{Jordan}(\cA)$ and $\Aut{\fA}$ are isomorphic, and these groups are contravariantly isomorphic to the group $\Aut(\SigA)$ of automorphisms of the spectral presheaf $\SigA$ of $\cA$.
\end{corollary}

\section{Outlook}
If we regard the spectral presheaf $\SigN$ as a generalised Gelfand spectrum of the von Neumann algebra $\cN$, we see from Thm. \ref{Thm_VNAs} that the spectral presheaf $\SigN$ of a von Neumann algebra contains enough information to determine $\cN$ as a $JBW$-algebra. What is missing for a full reconstruction of $\cN$ from its `spectrum' $\SigN$ is the Lie algebra structure on the self-adjoint elements of $\cN$: the Jordan product $(\hA,\hB)\mt\hA\cdot\cB=\frac{1}{2}(\hA\hB+\hB\hA)$ is the symmetrisation of the noncommutative product $(\hA,\hB)\mt\hA\hB$ (the factor $\frac{1}{2}$ in the Jordan product is just a matter of convention), and the Lie product $(\hA,\hB)\mt[\hA,\hB]=\hA\hB-\hB\hA$ is its anti-symmetrisation, so from knowledge of both, one can reconstruct the noncommutative product and hence the von Neumann algebra $\cN$. 

In \cite{Doe13x}, we will show how the Lie algebra structure is also encoded geometrically in the spectral presheaf. This will involve the theory of orientations on operator algebras and their state spaces pioneered by Connes \cite{Con74} and developed further by Alfsen, Hanche-Olsen, Iochum and Shultz, see \cite{AlfShu98,AlfShu01,AlfShu03} and references therein. Combining this with results on flows on the spectral presheaf \cite{Doe12c}, one can characterise orientations (more) geometrically and hence can reconstruct the Lie algebra structure on the self-adjoint elements $\cN_{\sa}$ of a von Neumann algebra, and by linear extension on all of $\cN$.

Thm. \ref{Thm_SpecPresheafOftenDeterminesUnitalCStarAlg} shows that for a certain class of unital $C^*$-algebras, the spectral presheaf also determines the Jordan $*$-structure. We will discuss flows on spectral presheaves of such algebras and results on Lie algebra structure in \cite{Doe13x} as well.


\textbf{Acknowledgements.} Discussions with Chris Isham, Boris Zilber and Yuri Manin are gratefully acknowledged. Rui Soares Barbosa provided key input, and Ralf Meyer, Tom Woodhouse, Nadish de Silva, Dan Marsden and Carmen Constantin gave valuable feedback for which I thank them. I also thank Bertfried Fauser, who carefully read a draft and made a number of very helpful remarks. Jan Hamhalter provided inspiration and help with a technical point, for which I am grateful. I also thank John Harding, Izumi Ojima, John Maitland Wright and Dirk Pattinson for their interest in this work (and for their patience).

\end{document}